\documentclass[11pt]{amsart}

\usepackage{amsmath,amsthm,amsfonts,amssymb,mathrsfs}
\usepackage[margin=1in]{geometry}
\usepackage{verbatim,mdwlist,setspace}
\usepackage[all,cmtip]{xy}
\usepackage[bookmarks]{hyperref}
\usepackage{cite}

\newenvironment{enumalph}{\begin{enumerate}  }{\end{enumerate}}

\newcommand{\subgrp}[1]{\langle #1 \rangle}
\newcommand{\set}[1]{\left\{ #1 \right\}}
\newcommand{\abs}[1]{\left| #1 \right|}

\newcommand{\wh}[1]{\widehat{ #1 }}
\newcommand{\ol}[1]{\overline{#1}}

\newcommand{\rank}{\operatorname{rank}}
\newcommand{\chr}{\operatorname{char}}

\newcommand{\End}{\operatorname{End}}
\newcommand{\Ext}{\operatorname{Ext}}

\newcommand{\Hom}{\operatorname{Hom}}
\newcommand{\ind}{\operatorname{ind}}
\newcommand{\Lie}{\operatorname{Lie}}
\newcommand{\Maxspec}{\operatorname{MaxSpec}}

\newcommand{\res}{\operatorname{res}}

\newcommand{\Ad}{\operatorname{Ad}}

\newcommand{\opH}{\operatorname{H}}

\newcommand{\hy}{\operatorname{hy}}
\newcommand{\id}{\operatorname{id}}

\newcommand{\Deltatau}{{}^\tau \Delta}
\newcommand{\Stau}{{}^\tau S}
\newcommand{\etau}{{}^\tau \varepsilon}

\newcommand{\F}{\mathbb{F}}
\newcommand{\N}{\mathbb{N}}
\newcommand{\Q}{\mathbb{Q}}
\newcommand{\R}{\mathbb{R}}

\newcommand{\Z}{\mathbb{Z}}

\newcommand{\calN}{\mathcal{N}}

\newcommand{\calV}{\mathcal{V}}

\newcommand{\frakb}{\mathfrak{b}}

\newcommand{\fraku}{\mathfrak{u}}

\newcommand{\g}{\mathfrak{g}}

\newcommand{\upA}{\mathsf{A}}
\newcommand{\upB}{\mathsf{B}}

\newcommand{\scrA}{\mathscr{A}}

\newcommand{\Uz}{U_\zeta}
\newcommand{\Uzo}{U_\zeta^0}
\newcommand{\Uzb}{U_\zeta(B)}
\newcommand{\Uzbp}{U_\zeta(B^+)}
\newcommand{\Uzbr}{U_\zeta(B_r)}
\newcommand{\Uzg}{U_\zeta(\g)}
\newcommand{\Uztr}{U_\zeta(T_r)}
\newcommand{\Uzn}{U_\zeta(N)}
\newcommand{\Uztn}{U_\zeta(TN)}
\newcommand{\Uztgr}{U_\zeta(TG_r)}
\newcommand{\Uztbr}{U_\zeta(TB_r)}
\newcommand{\Uzgr}{U_\zeta(G_r)}
\newcommand{\Uzurm}{U_\zeta(U_{r,m})}

\newcommand{\Uzur}{U_\zeta(U_r)}

\newcommand{\uz}{u_\zeta}
\newcommand{\uzg}{u_\zeta(\g)}
\newcommand{\uzu}{u_\zeta(\fraku)}
\newcommand{\uzup}{u_\zeta(\fraku^+)}

\newcommand{\uzb}{u_\zeta(\frakb)}
\newcommand{\uzbp}{u_\zeta(\frakb^+)}

\newcommand{\uzo}{u_\zeta^0}

\newcommand{\whZr}{\wh{Z}_r}

\newcommand{\Vg}{\calV_{\uzg}}
\newcommand{\Vb}{\calV_{\uzb}}
\newcommand{\Vbp}{\calV_{\uzbp}}

\numberwithin{equation}{subsection}

\newtheorem{theorem}{Theorem}[subsection]
\newtheorem{proposition}[theorem]{Proposition}

\newtheorem{corollary}[theorem]{Corollary}

\newtheorem{lemma}[theorem]{Lemma}

\newtheorem*{theorem*}{Theorem}
\newtheorem*{lemma*}{Lemma}

\theoremstyle{definition}
\newtheorem{assumption}[theorem]{Assumption}

\newtheorem{remark}[theorem]{Remark}
\newtheorem{remarks}[theorem]{Remarks}

\title[On injective modules and support varieties for the small quantum group]{On injective modules and support varieties \\ for the small quantum group}

\thanks{This paper previously titled \emph{Injectivity criteria and support varieties for the small quantum group}}

\author{Christopher M.\ Drupieski}

\address{
Department of Mathematics \\
University of Georgia \\
Athens, GA~30602-7403}
\email{cdrup@math.uga.edu}


\begin{document}

\begin{abstract}
Let $\uzg$ denote the ``small quantum group'' associated to the simple complex Lie algebra $\g$, with parameter $q$ specialized to a primitive $\ell$-th root of unity $\zeta$ in the field $k$. Generalizing a result of Cline, Parshall and Scott, we show that if $M$ is a finite-dimensional $\uzg$-module admitting a compatible torus action, then the injectivity of $M$ as a module for $\uzg$ can be detected by the restriction of $M$ to certain ``root subalgebras'' of $\uzg$. If $\chr(k)= p > 0$, then this injectivity criterion also holds for the higher Frobenius--Lusztig kernels $U_\zeta(G_r)$ of the quantized enveloping algebra $U_\zeta(\g)$. Now suppose that $M$ lifts to a $U_\zeta(\g)$-module. Using a new rank variety type result for the support varieties of $\uzg$, we prove that the injectivity of $M$ for $\uzg$ can be detected by the restriction of $M$ to a single root subalgebra.
\end{abstract}

\maketitle

\section*{Introduction}

Let $k$ be an algebraically closed field of characteristic $p > 0$, and let $G$ be the semisimple, simply-connected algebraic group over $k$ with root system $\Phi$. Fix a maximal torus $T \subset G$, defined and split over the prime field $\F_p$, such that $\Phi$ is the root system of $T$ in $G$. For $\alpha \in \Phi$, let $U_\alpha$ be the corresponding one-dimensional root subgroup in $G$. Set $U = \subgrp{U_\alpha:\alpha \in \Phi^-}$, the unipotent radical of the Borel subgroup $B = TU \subset G$. Let $F: G \rightarrow G$ denote the Frobenius morphism, and for $H$ a closed $F$-stable subgroup of $G$, let $H_r$ denote the (scheme-theoretic) kernel of the $r$-th iterate $F^r$ of the Frobenius morphism $F|_H: H \rightarrow H$.

In their paper \cite{Cline:1985}, Cline, Parshall and Scott showed that a finite-dimensional $TG_r$-module $M$ is injective if and only if its restriction $M|_{U_{\alpha,r}}$ is injective for each root $\alpha \in \Phi$. They proved this criterion by reducing to the case of a finite-dimensional $TB_r$-module $M$, and then arguing by induction on the dimension of $M$. The induction argument was combinatorial in nature, and relied on a well-chosen ordering for the positive roots in $\Phi$. Later, Friedlander and Parshall \cite{Friedlander:1986b,Friedlander:1987} deduced a geometric proof of the injectivity criterion in the special case $r=1$ by studying the support variety $\calV_{G_1}(M)$, a conical subvariety of the affine space $\calN$ of nilpotent elements in $\g := \Lie(G)$. Their breakthrough was to provide a rank variety type interpretation for $\calV_{G_1}(M)$. In particular, they determined that $M|_{U_{\alpha,1}}$ is injective if and only if the corresponding root vector $f_\alpha \in \g$ is not an element of $\calV_{U_1}(M)$. Since $\calV_{U_1}(M)$ is naturally a $T$-space (by the assumption of the $T$-action on $M$), and since any $T$-stable subvariety of $\fraku$ must contain a root vector, this proved the criterion.

Now let $k$ be an algebraically closed field of characteristic $p \neq 2$, and let $\ell$ be an odd positive integer. If $\Phi$ has type $G_2$, then assume also that $p \neq 3$ and that $\ell$ is coprime to 3. Let $U_\zeta(\g)$ be the quantized enveloping algebra (Lusztig form) associated to $\g$, with parameter $q$ specialized to a primitive $\ell$-th root of unity $\zeta \in k$. Let $\uzg$ be the ``small quantum group,'' the Hopf algebraic kernel of the quantum Frobenius morphism $F_\zeta: U_\zeta(\g) \rightarrow \hy(G)$. In this paper we generalize the results of \cite{Cline:1985} and \cite{Friedlander:1986b,Friedlander:1987} to deduce criteria for detecting the injectivity of modules over $u_\zeta(\g)$ that admit compatible actions by the quantum torus $U_\zeta^0 \subset U_\zeta(\g)$. Our first main result (Theorem \ref{theorem:injectivitycriterion}) is that a finite-dimensional rational $\Uzo \uzg$-module $M$ is injective if and only if the restriction $M|_{u_\zeta(f_\alpha)}$ is injective for each root subalgebra $u_\zeta(f_\alpha) \subset \uzg$ ($\alpha \in \Phi$). If $p > 0$, then the injectivity of a module $M$ for the higher Frobenius--Lusztig kernels $\Uzgr$ of $\Uzg$ can also be detected by the restriction of $M$ to the higher Frobenius--Lusztig kernels $U_\zeta(U_{\alpha,r})$ of the root subalgebras. The problem of identifying sublagebras of a given Hopf algebra that detect injectivity in the above manner is a topic of historical and current interest \cite{Quillen:1971,Friedlander:2007,Cline:1985,Towers:2009}.

Next, we study cohomological support varieties for $\uzg$ and its Borel subalgebras $\uzb$ and $\uzbp$. It is known (in characteristic zero) by results of Ginzburg and Kumar \cite{Ginzburg:1993} and (in positive characteristics) by results of the author \cite{Drupieski:2009} that the $\uzg$-support variety $\calV_{\uzg}(M)$ of a $\uzg$-module $M$ identifies naturally with a conical subvariety of the nullcone $\calN \subset \g$, and that the $\uzb$-support variety $\calV_{\uzb}(M)$ identifies naturally with a conical subvariety of $\fraku:=\Lie(U)$. The structure of $\Vg(M)$ is known, for example, if $M = H_\zeta^0(\lambda)$ is an induced module, by work of Ostrik \cite{Ostrik:1998} and of Bendel, Nakano, Parshall and Pillen \cite{Bendel:2009}, and if $M = L_\zeta(\lambda)$ is a simple module, by work of the author, Nakano and Parshall \cite{Drupieski:2009a}. Beyond these special cases few explicit calculations are known. A primary obstruction to computing $\Vg(M)$ for arbitrary $M$ is the lack of a theory of rank varieties for $\uzg$-modules, or, more generally, a theory of rank varieties for arbitrary finite-dimensional Hopf algebras. Some partial progress in this direction has been made; see, for example, \cite{Benson:2007,Pevtsova:2006,Towers:2009}.

In this paper we contribute to the development of a theory of rank varieties for the Borel subalebra $\uzb$ by proving a partial generalization of Friedlander and Parshall's result on the support varieties of $p$-unipotent restricted Lie algebras. Specifically, in Section \ref{section:Bsupportvarieties} we show that the restriction $M|_{u_\zeta(f_\alpha)}$ is injective if and only if the root vector $f_\alpha \in \g$ is not contained in $\calV_{\uzb}(M)$. This enables us in Section \ref{subsection:secondproof} to provide a second, geometric proof of the $r=0$ version of Theorem \ref{theorem:injectivitycriterion}. Finally, in Section \ref{subsection:Gsupportvarieties} we consider the case when $M$ is a $\Uzg$-module. Then $\Vg(M)$ is naturally a $G$-stable subvariety of the nullcone $\calN$. In this case, we use the structure of nilpotent orbits in $\calN$ to deduce that $M$ is injective for $\uzg$ if and only if the restriction $M|_{u_\zeta(f_{\alpha_h})}$ is injective for the root subalgebra corresponding to the highest long root $\alpha_h \in \Phi^+$.

It is our hope that partial rank variety results in Section \ref{section:Bsupportvarieties} could be extended to support varieties for the entire small quantum grouop $\uzg$. We also hope that rank variety results might help answer the question of naturality between support varieties over $\uzg$ and $\uzb$, that is, given a finite-dimensional $\uzg$-module $M$, whether the intersection $\Vg(M) \cap \fraku$ is equal to the support variety $\Vb(M)$. This is an issue which is easily settled in the affirmative for Frobenius kernels of algebraic groups using rank varieties, but which remains open for small quantum groups.

\section{Preliminaries}

\subsection{Quantized enveloping algebras} \label{subsection:QEAs}

Let $\Phi$ be a finite, indecomposable root system. Fix a set of simple roots $\Pi \subset \Phi$, and let $\Phi^+$, $\Phi^-$ be the corresponding sets of positive and negative roots in $\Phi$. Let $\Z\Phi$ denote the root lattice of $\Phi$, $X$ the weight lattice of $\Phi$, and $X^+ \subset X$ the subset of dominant weights. Write $X_{p^r \ell} = \set{\lambda \in X^+: (\lambda,\beta^\vee) < p^r \ell \;\; \forall \, \beta \in \Pi}$ for the set of $p^r\ell$-restricted dominant weights. Let $W$ denote the Weyl group of $\Phi$. It is generated by the simple reflections $\set{s_\beta:\beta \in \Pi}$. The root system $\Phi$ spans a real vector space $\mathbb{E}$, possessing a positive definite, $W$-invariant inner product $(\cdot,\cdot)$, normalized so that $(\alpha,\alpha) = 2$ if $\alpha \in \Phi$ is a short root.

Let $k$ be a field of characteristic $p \neq 2$ (and $p \neq 3$ if $\Phi$ has type $G_2$). Let $q$ be an indeterminate over $k$. Then the quantized enveloping algebra $U_q = U_q(\g)$ is the $k(q)$-algebra defined by generators $\set{E_\alpha,F_\alpha,K_\alpha,K_\alpha^{-1}:\alpha \in \Pi}$ and relations as in \cite[Chapter 4]{Jantzen:1996}. Multiplication in $U_q$ induces vector space isomorphisms $U_q \cong U_q^- \otimes U_q^0 \otimes U_q^+ \cong U_q^+ \otimes U_q^0 \otimes U_q^-$ (triangular decomposition), where $U_q^-$ (resp.\ $U_q^0,U_q^+$) is the $k(q)$-subalgebra of $U_q$ generated by $\set{F_\alpha : \alpha \in \Pi}$ (resp.\ $\set{K_\alpha,K_\alpha^{-1}:\alpha \in \Pi}$, $\set{E_\alpha: \alpha \in \Pi}$). The algebra $U_q$ is a Hopf algebra, with Hopf algebra structure maps defined in \cite[\S4.8]{Jantzen:1996}. The algebras $U_q^+U_q^0$ and $U_q^-U_q^0$ are Hopf subalgebras of $U_q$, but $U_q^+$ and $U_q^-$ are not.

Let $\ell \in \N$ be an odd positive integer, with $\ell$ coprime to $3$ if $\Phi$ has type $G_2$. Fix a primitive $\ell$-th root of unity $\zeta \in k$. Set $\upA = k[q,q^{-1}]$. Then $k$ is naturally an $\upA$-module under the specialization $q \mapsto \zeta$. For $n \in \N$, the divided powers $E_\alpha^{(n)}$, $F_\alpha^{(n)}$ are defined in \cite[\S8.6]{Jantzen:1996}. Let $U_\upA$ be the $\upA$-subalgebra of $U_q$ generated by
\[
\set{E_\alpha^{(n)},F_\alpha^{(n)},K_\alpha,K_\alpha^{-1}:\alpha \in \Pi, n \in \N}.
\]
Set $U_k = U_\upA \otimes_\upA k$. Then the algebra $U_\zeta = U_\zeta(\g)$ is defined as the quotient of $U_k$ by the two-sided ideal $\subgrp{K_\alpha^\ell \otimes 1 - 1 \otimes 1 : \alpha \in \Pi}$. By abuse of notation, we denote the generators $E_\alpha,F_\alpha,K_\alpha,K_\alpha^{-1} \in U_q$ as well as their images in $U_\zeta$ by the same symbols. The algebra $U_\zeta$ inherits a triangular decomposition from $U_q$.

Fix a choice of reduced expression $w_0 = s_{\beta_1} s_{\beta_2} \cdots s_{\beta_N}$ ($\beta_i \in \Pi$) for the longest word $w_0 \in W$. So $N = \abs{\Phi^+}$. For $1 \leq i \leq N$, set $w_i = s_{\beta_1} \cdots s_{\beta_{i-1}}$ (so $w_1 = 1$), and set $\gamma_i = w_i(\beta_i)$. Then $\set{\gamma_1,\ldots,\gamma_N}$ is a convex ordering of the positive roots in $\Phi^+$, that is, if $\gamma_i + \gamma_j = \gamma_l$ with $i < j$, then $i < l < j$. Now for $\alpha \in \Pi$, let $T_\alpha$ denote Lusztig's automorphism of $U_q$, as defined in \cite[Chapter 8]{Jantzen:1996}. Since the $T_\alpha$ satisfy the braid group relations for $W$, there exists for each $w \in W$ a well-defined automorphism $T_w$ of $U_q$. For each $1 \leq i \leq N$, define the root vetors $E_{\gamma_i} = T_{w_i}(E_{\beta_i})$ and $F_{\gamma_i} = T_{w_i}(F_{\beta_i})$. Then the collection of monomials $\set{E_{\gamma_1}^{a_1} \cdots E_{\gamma_N}^{a_N}: a_i \in \N}$ (resp.\ $\set{F_{\gamma_1}^{a_1} \cdots F_{\gamma_N}^{a_N}: a_i \in \N}$) forms a PBW-type basis for $U_q^+$ (resp.\ $U_q^-$). Replacing the root vectors by their divided powers, we obtain $\upA$-bases
\begin{align}
& \set{E_{\gamma_1}^{(a_1)} \cdots E_{\gamma_N}^{(a_N)}: a_i \in \N}
\qquad \text{and} \label{eq:Edivpowerbasis} \\
& \set{F_{\gamma_1}^{(a_1)} \cdots F_{\gamma_N}^{(a_N)}: a_i \in \N} \label{eq:divpowerbases}
\end{align}
for the algebras $U_\upA^+$ and $U_\upA^-$, respectively, which project onto $k$-bases for $U_\zeta^+$ and $U_\zeta^-$.

The following lemma, describing ``commutation'' relations between the root vectors in $U_q$, generalizes an observation of Levendorskii and Soibelman \cite{Levendorskii:1991}.

\begin{lemma} \label{lemma:commutationrelations}
Let $S \subset \upA = k[q,q^{-1}]$ be the multiplicatively closed set generated by
\begin{align*}
&\set{1} & \text{if $\Phi$ has type $ADE$}, \\
&\set{q^2-q^{-2}} & \text{if $\Phi$ has type $BCF$}, \\
&\set{q^2-q^{-2}, q^3-q^{-3}} & \text{if $\Phi$ has type $G_2$}.
\end{align*}
Set $\scrA = S^{-1} \upA$, and let $1 \leq i < j \leq N$. Then in $U_q$ we have
\begin{enumalph}
\item $E_{\gamma_i}E_{\gamma_j} = q^{(\gamma_i,\gamma_j)}E_{\gamma_j} E_{\gamma_i} + (*)$, where $(*)$ is an $\scrA$-linear combination of monomials $E_{\gamma_1}^{a_1} \cdots E_{\gamma_N}^{a_N}$ with $a_s = 0$ unless $i < s < j$.
\item $F_{\gamma_i}F_{\gamma_j} = q^{(\gamma_i,\gamma_j)} F_{\gamma_j} F_{\gamma_i} + (*)$, where $(*)$ is an $\scrA$-linear combination of monomials $F_{\gamma_1}^{a_1} \cdots F_{\gamma_N}^{a_N}$ with $a_s = 0$ unless $i < s < j$.
\end{enumalph}
\end{lemma}

\begin{proof}
Parts (a) and (b) are equivalent: apply the algebra automorphism $\omega$ defined below in \S \ref{subsection:FLkernels}. When $\g$ has rank two, part (a) can be verified by direct calculation, using, for example, the QuaGroup package of the computer program GAP (cf.\ also \cite{Lusztig:1990a}). From the rank two case the result is deduced for arbitrary $\g$ by the arguments in the proof of \cite[Theorem 9.3(iv)]{De-Concini:1993}.
\end{proof}

\begin{remarks}
(1) Versions of Lemma \ref{lemma:commutationrelations} appear in the literature \cite{De-Concini:1990,De-Concini:1993} with our choice for the ring $\scrA$ replaced by $\Q[q,q^{-1}]$. Direct calculation in types $B_2$ and $G_2$ shows that such a formulation is incorrect and that the extra denominators are necessary whenever $\Phi$ has two root lengths. Since the generators of $S$ do not map to zero under the specialization $q \mapsto \zeta$, we get that relations of Lemma \ref{lemma:commutationrelations} also hold for $\Uz$.

(2) According to \cite[Theorem 2.4]{Xi:1999}, any permutation of the ordering of the root vectors in \eqref{eq:Edivpowerbasis} also yields an $\upA$-basis for $U_q^+$, and similarly for \eqref{eq:divpowerbases} and $U_q^-$.
\end{remarks}

\begin{comment}
Using Lemma \ref{lemma:commutationrelations}, one can show that for any permutation $\sigma$ of $\set{1,\ldots,N}$, the sets
\begin{equation}
\set{E_{\gamma_{\sigma(1)}}^{(a_1)} \cdots E_{\gamma_{\sigma(N)}}^{(a_N)}: a_i \in \N} \qquad \text{and} \qquad \set{F_{\gamma_{\sigma(1)}}^{(a_1)} \cdots F_{\gamma_{\sigma(N)}}^{(a_N)}: a_i \in \N}
\end{equation}
are also $\upA$-bases for $U_\upA^+$ and $U_\upA^-$, respectively, hence project to $k$-bases for $U_\zeta^+$ and $U_\zeta^-$.
\end{comment}

\subsection{Frobenius--Lusztig kernels} \label{subsection:FLkernels}

The elements $\set{E_\alpha,F_\alpha,K_\alpha: \alpha \in \Pi}$ of $\Uz$ generate a finite-dimensional Hopf subalgebra of $\Uz$, called the small quantum group and denoted by $u_\zeta = \uzg$. It is a normal Hopf subalgebra of $\Uz$, and the Hopf algebraic quotient $\Uz//\uz$ is isomorphic to $\hy(G)$, the hyperalgebra of the semisimple, simply-connected algebraic group $G$ over $k$ with root system $\Phi$. The quotient map $F_\zeta: \Uz \rightarrow \hy(G)$ was constructed by Lusztig \cite[\S8.10--8.16]{Lusztig:1990a}, and is called the quantum Frobenius morphism. For this reason, the algebra $\uz$ is also called the Frobenius--Lusztig kernel of $\Uzg$. The subalgebra $\uzo:= \uzg \cap \Uzo$ is isomorphic to the group ring $k(\Z_\ell)^{\rank(\g)}$. Since we assumed $\zeta$ to be a primitive $\ell$-th root of unity in the field $k$ of characteristic $p$, we must have $p \nmid \ell$. Then $\uzo$ is a semisimple algebra.

Fix $r \in \N$, and suppose $p> 0$. Define $U_\zeta(G_r)$ to be the Hopf-subalgebra of $\Uz$ generated by
\begin{equation} \label{eq:kernelgenerators}
\set{E_\alpha,E_\alpha^{(p^i \ell)}, F_\alpha, F_\alpha^{(p^i \ell)}, K_\alpha : \alpha \in \Pi, 0 \leq i \leq r-1} \subset \Uzg.
\end{equation}
Then $U_\zeta(G_r)$ is a finite-dimensional subalgebra of $\Uz$, and $F_\zeta(U_\zeta(G_r)) = \hy(G_r)$, the hyperalgebra of the $r$-th Frobenius kernel of $G$. We call $U_\zeta(G_r)$ the $r$-th Frobenius--Lusztig kernel of $\Uz$, and collectively we refer to the $U_\zeta(G_r)$ with $r \geq 1$ as the higher Frobenius--Lusztig kernels of $\Uz$. If $r=0$, then $U_\zeta(G_r)$ reduces to $u_\zeta$, the small quantum group. The higher Frobenius--Lusztig kernels of $\Uz$ are defined only if $p = \chr(k) > 0$. Indeed, if $\chr(k) = 0$, then the algebra generated by the set \eqref{eq:kernelgenerators} is all of $\Uz$. Since most of the results presented in this paper are characteristic-independent, we have adopted the position of stating results whenever possible for the higher Frobenius--Lusztig kernels $\Uzgr$ of $\Uz$, with the understanding that the reader should take $r=0$ whenever $p = 0$.

We will be concerned with certain distinguished subalgebras of $\Uz$ corresponding to the subgroup schemes $U_r$, $B_r$, and $TG_r$ of $G$. Define $U_\zeta(B) = U_\zeta^0 U_\zeta^-$, $\Uzur = U_\zeta^- \cap \Uzgr$, $\Uzbr = U_\zeta(B) \cap \Uzgr$, and $\Uztgr = U_\zeta^0 \Uzgr$. Then, for example, the algebra $\Uzur$ admits a basis consisting of all monomials \eqref{eq:divpowerbases} with $0 \leq a_i < p^r \ell$. If $r=0$, then write $\uzb$, $\uzu$, and $\Uzo \uzg$ for $\Uzbr$, $\Uzur$, and $\Uztgr$, respectively. We will also be concerned with certain subalgebras of $\Uzur$ generated by root vectors. For each $\alpha \in \Phi^+$, define $U_\zeta(U_{\alpha,r})$ to be the subalgebra of $\Uzur$ generated by the elements
\[
\set{ F_\alpha, F_\alpha^{(\ell)},F_\alpha^{(p\ell)},\ldots,F_\alpha^{(p^{r-1}\ell)} }.
\]
Then $U_\zeta(U_{\alpha,r})$ admits a basis consisting of all divided powers $F_\alpha^{(n)}$ with $0 \leq n < p^r \ell$. As an algebra, $U_\zeta(U_{\alpha,r})$ is isomorphic to the truncated polynomial ring
\[
k[Y,X_0,X_1,\ldots,X_{r-1}]/(Y^\ell,X_0^p,X_1^p,\ldots,X_{r-1}^p).
\]
The isomorphism maps $F_\alpha \mapsto Y$ and $F_\alpha^{(p^i \ell)} \mapsto X_i$. If $r=0$, then $U_\zeta(U_{\alpha,r})$ reduces to the subalgebra $u_\zeta(f_\alpha)$ of $\uz$ generated by $F_\alpha$. Finally, for each $1 \leq m \leq N$, define $\Uzurm$ to be the subspace of $\Uzur$ spanned by the monomial basis vectors \eqref{eq:divpowerbases} with $0 \leq a_i < p^r\ell$ and $a_i = 0$ for $i > m$. It follows from Lemma \ref{lemma:commutationrelations} that $\Uzurm$ is a normal subalgebra of $U_\zeta(U_{r,m+1})$. Of course, $U_\zeta(U_{r,N}) = \Uzur$.

There exists an involutory $k(q)$-algebra automorphism $\omega$ of $U_q$ defined by
\begin{equation} \label{eq:omega}
\omega(E_\alpha) = F_\alpha, \quad \omega(F_\alpha) = E_\alpha, \quad \omega(K_\alpha) = K_\alpha^{-1} \quad (\alpha \in \Pi).
\end{equation}
For each $\gamma \in \Phi^+$, $\omega(E_\gamma) = \pm q^a F_\gamma$ for some $a \in \Z$ (depending on $\gamma$) \cite[8.14(9)]{Jantzen:1996}. The automorphism $\omega$ descends to an automorphism of $\Uz$ and of its Frobenius--Lusztig kernels. Now define ``positive'' versions of the distinguished subalgebras of $\Uz$ by setting $U_\zeta(U_r^+) = \omega(\Uzur)$, $U_\zeta(B_r^+) = \omega(\Uzbr)$, $U_\zeta(U_{\alpha,r}^+) = \omega(U_\zeta(U_{\alpha,r})$, and so on. If $r=0$, then denote $U_\zeta(U_{\alpha,r}^+)$ by $u_\zeta(e_\alpha)$. Collectively, we denote the collection of all (positive and negative) root subalgebras by writing $U_\zeta(U_{\alpha,r})$, $\alpha \in \Phi$ (i.e., by not distinguishing between $\alpha$ being a positive or an arbitrary root).

\subsection{Hopf algebra actions on cohomology} \label{subsection:Hopfalgebraactions}

Let $H$ be a Hopf algebra with comultiplication $\Delta$ and antipode $S$. Given $h \in H$, write $\Delta(h) = \sum h^{(1)} \otimes h^{(2)}$ (Sweedler notation). Later we may omit the summation symbol, and just write $\Delta(h) = h^{(1)} \otimes h^{(2)}$. The left and right adjoint actions of $H$ on itself are defined for $h,u \in H$ by
\begin{equation*}
\Ad_\ell(h)(u) = \sum h^{(1)} w S(h^{(2)}) \qquad \text{and} \qquad \Ad_r(h)(u) = \sum S(h^{(1)}) u h^{(2)}.
\end{equation*}
Now let $V$ be an $H$-module, and let $A$ be an $\Ad_r$-stable subalgebra of $H$. Let $\upB_\bullet(A)$ denote the left bar resolution for $A$, defined in \cite[X.2]{Mac-Lane:1995}. We denote a typical element of $\upB_n(A)$ by $a[a_1,\ldots,a_n]$. The right $\Ad_r$-action of $H$ on $A$ extends diagonally to a right action of $H$ on $\upB_\bullet(A)$. The Hopf algebra axioms for $H$ guarantee that $\upB_\bullet(A)$ is thus a complex of right $H$-modules.

The cohomology groups $\opH^\bullet(A,V)$ can be computed as the homology of the cochain complex $C^\bullet(A,V):=\Hom_A(\upB_\bullet(A),V)$. Define an action of $H$ on $C^\bullet(A,V)$ by setting, for $h \in H$, $u \in \upB_n(A)$, and $f \in C^n(A,V)$,
\begin{equation*}
(h.f)(u) = \sum h^{(1)}.f(u \cdot h^{(2)}).
\end{equation*}
This definition makes sense because $V$ was assumed to be an $H$-module. Using the Hopf algebra axioms for $H$, one can easily show that this action makes $C^\bullet(A,V)$ a complex of left $H$-modules, hence that there exists an induced action of $H$ on $\opH^\bullet(A,V)$. We call the $H$-action thus obtained the adjoint action of $H$ on $\opH^\bullet(A,V)$.

\begin{remark}
Suppose that $A$ is an $\Ad_\ell$-stable subalgebra of $H$. Then $\upB_\bullet(A)$ is naturally a complex of left $H$-modules, and $\Hom_k(\upB_\bullet(A),V)$ is a left $H$-module under the ``usual'' diagonal action of $H$, defined in \eqref{eq:Hopfdiagonalaction}. This is, however, the wrong approach to take when defining an action of $H$ on $\opH^\bullet(A,V)$, because it is not clear in general that the usual diagonal action of $H$ on $\Hom_k(\upB_\bullet(A,V)$ stabilizes the subspace of $A$-homomorphisms, nor that it commutes with the differential of the complex $C^\bullet(A,V)$. 
\end{remark}

For future reference it will be useful to note:

\begin{lemma} \textup{\cite[Corollary 3.14]{Drupieski:2009}}
The small quantum group $\uzg$ is stable under the right adjoint action of $\Uz$ on itself. The Borel subalgebra $\uzb$ (resp.\ $\uzbp$) is stable under the right adjoint action of $U_\zeta(B)$ (resp.\ $U_\zeta(B^+)$) on itself.
\end{lemma}

\section{Representation theory}

\subsection{Ordinary representation theory} \label{subsection:ordinaryreptheory}

Denote the algebra $\Uzurm$ defined in \S \ref{subsection:FLkernels} by $A_m$. Using Lemma \ref{lemma:commutationrelations}, one can show that for any $A_m$-module $M$, the space of invariants $M^{A_m}$ is non-zero. This implies that, up to isomorphism, there exists a unique irreducible (left or right) $A_m$-module, namely, the trivial module $k$. The space of invariants $(A_m)^{A_m}$ for the (left or right) regular action of $A_m$ on itself is one-dimensional, spanned by the vector
\begin{equation} \label{eq:integral}
\textstyle \int_m:=F_{\gamma_1}^{(p^r\ell - 1)} \cdots F_{\gamma_m}^{(p^r \ell - 1)}.
\end{equation}
(The notation in \eqref{eq:integral} is meant to remind the reader of the two-sided integral in the finite-dimensional Hopf algebra $\hy(U_r)$.) Then the regular module $A_m$ is indecomposable, hence it is the (left and right) projective cover for the trivial module $k$. Trivially, the dual of every irreducible left (resp.\ right) $A_m$-module is an irreducible right (resp.\ left) $A_m$-module, so the (left or right) regular module $A_m$ is also injective \cite[Theorem 58.6]{Curtis:2006}. Then the regular module $A_m$ is also the (left and right) injective envelope of the trivial module $k$. From this discussion we conclude:

\begin{lemma} \label{lemma:Amfree}
Retain the notation of the previous paragraph.
\begin{enumalph}
\item A finite-dimensional $A_m$-module is projective iff it is injective iff it is free.
\item There exists an isomorphism of left $A_m$-modules $A_m \cong (A_m)^* := \Hom_k(A_m,k)$, where the left action of $A_m$ on $(A_m)^*$ is induced by the right multiplication of $A_m$ on itself.
\end{enumalph}
\end{lemma}

The next lemma is similar to \cite[II.9.4]{Jantzen:2003}.

\begin{lemma} \label{lemma:borelunipotentequivalence}
Let $M$ be a finite-dimensional $\Uzbr$-module. The following are equivalent:
\begin{enumerate}
\item $M$ is injective for $\Uzbr$.
\item $M$ is injective for $\Uzur$.
\item $M$ is projective for $\Uzur$.
\item $M$ is projective for $\Uzbr$.
\end{enumerate}
\end{lemma}

\begin{proof}
Statements (1) and (4) are equivalent because $\Uzbr$ is a finite-dimensional Hopf algebra, cf.\ \cite{Larson:1969} and \cite[\S 61--62]{Curtis:2006}. Statements (2) and (3) are equivalent by the $m=N$ case of Lemma \ref{lemma:Amfree}. The implication (4) $\Rightarrow$ (3) follows from the fact that $\Uzbr$ is free as a left module over $\Uzur$. Now suppose that $M$ is projective for $\Uzur$, and let $W$ be an arbitrary $\Uzbr$-module. Considering the weight space decomposition of $\Hom_{\Uzur}(M,W)$ for the diagonal action of $U_\zeta(T_r) := \Uzbr \cap U_\zeta^0$, we get
\begin{equation} \label{eq:homweightdecomposition}
\Hom_{\Uzur}(M,W) = \bigoplus_{\lambda \in X_{p^r\ell}} \Hom_{\Uzur}(M,W)_\lambda \cong \bigoplus_{\lambda \in X_{p^r\ell}} \Hom_{\Uzbr}(M \otimes \lambda,W).
\end{equation}
Since $\Hom_{\Uzur}(M,-)$ is an exact functor, then so must be each $\Hom_{\Uzbr}(M \otimes \lambda,-)$. In particular, $\Hom_{\Uzbr}(M,-)$ must be exact. This proves the implication (3) $\Rightarrow$ (4).
\end{proof}

\begin{remark} \label{remark:indecomposableinjectives}
A complete description of the indecomposable injective (equivalently, projective) $\Uzbr$-modules is provided as follows. Write the identity $1 \in \Uztr$ as a sum of primitive orthogonal idempotents: $1 = \sum_{\lambda \in X_{p^r\ell}} e_\lambda$. In the notation of \cite[\S 1.2]{Andersen:1991}, $u.e_\lambda = \chi_\lambda(u). e_\lambda$ for all $u \in \Uztr$. Then $\set{k e_\lambda: \lambda \in X_{p^r\ell}}$ is a complete set of non-isomorphic simple $\Uztr$-modules. Now write $\Uzbr = \bigoplus_{\lambda \in X_{p^r\ell}} \Uzur e_\lambda$. Each $\Uzur e_\lambda$ is injective and projective as a module for $\Uzbr$ by Lemma \ref{lemma:borelunipotentequivalence}. The $\Uzbr$-module $\Uzur e_\lambda$ has one-dimensional socle of weight $\lambda - 2(p^r\ell - 1)\rho$, where $\rho = \frac{1}{2} \sum_{\alpha \in \Phi^+} \alpha$, and one-dimensional head of weight $\lambda$. Then, as a $\Uzbr$-module, $\Uzur e_\lambda$ is the projective cover of $\lambda$ and the injective hull of $\lambda - 2(p^r\ell - 1)\rho$. Replacing $B_r$ and $U_r$ by $B_r^+$ and $U_r^+$, one obtains a similar description for the indecomposable injective $U_\zeta(B_r^+)$-modules.
\end{remark}

\subsection{Rational representation theory} \label{subsection:rationalreptheory}

We presuppose that the reader is familiar with the rational (i.e., integrable) representation theory of quantized enveloping algebras, as developed in the work of Andersen, Polo and Wen \cite{Andersen:1992a,Andersen:1992,Andersen:1991}. Recall that a $\Uzo$-module $M$ is rational if and only if it admits a weight space decomposition of the form $M = \bigoplus_{\lambda \in X} M_\lambda$. Given $N \in \set{B_r,G_r,U_{r,m}}$, a $U_\zeta(TN) := U_\zeta^0 U_\zeta(N)$-module $M$ is rational if and only if it is rational as a $\Uzo$-module. By definition, all $U_\zeta(N)$-modules are rational. We summarize below certain results on the rational representation theory of the Frobenius--Lusztig kernels that will be needed later in \S \ref{subsection:reductionBorel}.

Write $A_m = \Uzurm$, and write $TA_m$ for the algebra $U_\zeta(TU_{r,m}) = \Uzo \Uzurm$.

\begin{lemma} \label{lemma:weightbasis}
Let $M$ be a finite-dimensional rational $TA_m$-module, and suppose that $M$ is free as a module for $A_m$. Then there exists an $A_m$-basis for $M$ consisting of weight vectors for $\Uzo$.
\end{lemma}

\begin{proof}
The proof is by induction on the $A_m$-rank of $M$. To begin, suppose $M$ is free of rank one over $A_m$, with basis $\set{v} \subset M$. Write $v = \sum_{\lambda \in X} v_\lambda$, the decomposition of $v$ as a finite sum of $\Uzo$-weight vectors. Since $\int_m v \neq 0$, there exists $\lambda \in X$ with $\int_m v_\lambda \neq 0$. Fix such a $\lambda$. Then $A_m.v_\lambda \cong A_m$ as left $A_m$-modules. Indeed, there exists a natural surjective map $A_m \twoheadrightarrow A_m.v_\lambda$, which is nonzero on the one-dimensional socle $\int_m$ of $A_m$, hence must be an isomorphism. Then $A_m . v_\lambda$ is an $A_m$-submodule of $M$ of dimension $\dim A_m = \dim V$, so we must have $M = A_m.v_\lambda$.

Now let $n > 1$, and suppose $M$ is free of rank $n$ over $A_m$ with basis $\set{v_1,\ldots,v_n}$. Then $M = M' \oplus M''$, where $M'$ is the $A_m$-submodule generated by $\set{v_1}$, and $M''$ is the $A_m$-submodule generated by $\set{v_2,\ldots,v_n}$. By induction, there exist $A_m$-bases $S'$ and $S''$ for $M'$ and $M''$, respectively, consisting of weight vectors for $\Uzo$. Then the union $S = S' \cup S''$ is an $A_m$-basis for $M$ consisting of weight vectors for $\Uzo$.
\end{proof}

\begin{lemma} \label{lemma:indecomposableinjectives}
The indecomposable rationally injective $TA_m$-modules have the form
\[ Y_{\lambda,m} = \ind_{\Uzo}^{TA_m} \lambda = \Hom_{\Uzo}(TA_m,\lambda) \qquad (\lambda \in X). \]
As a module for $A_m$, $Y_{\lambda,m} \cong (A_m)^*$. In particular, $Y_{\lambda,m}$ is a free rank-one $A_m$-module, with basis consisting of a $\lambda$-weight vector for $\Uzo$.
\end{lemma}

\begin{proof}
The induction functor takes injective modules to injective modules, and by Frobenius reciprocity, $Y_{\lambda,m}$ has simple socle $\lambda$. This proves that the $Y_{\lambda,m}$ are the indecomposable injective modules for $TA_m$. The last statement follows from Lemmas \ref{lemma:weightbasis} and \ref{lemma:Amfree}.
\end{proof}

Taking $m=N$ in Lemma \ref{lemma:indecomposableinjectives}, we obtain a description of the indecomposable rationally injective $U_\zeta(TU_r) = \Uztbr$-modules. Now, given $\lambda \in X$, define rational left $\Uztgr$-modules by
\begin{align*}
\whZr(\lambda) &= \Uztgr \otimes_{U_\zeta(TB_r^+)} \lambda, \quad \text{and} \\
\whZr'(\lambda) &= \ind_{\Uztbr}^{\Uztgr} \lambda = \Hom_{\Uztbr}(\Uztgr,\lambda).
\end{align*}
The left $\Uztgr$-module structure of $\whZr'(\lambda)$ is induced by the right multiplication of $\Uztgr$ on itself. As a $\Uzur$-module, $\whZr(\lambda) \cong \Uzur \otimes \lambda$, and as a $U_\zeta(U_r^+)$-module, $\whZr'(\lambda) \cong \Hom_k(U_\zeta(U_r^+),\lambda)$, with the left action induced by the right multiplication of $U_\zeta(U_r^+)$ on itself. Lemma \ref{lemma:indecomposableinjectives} and its corresponding version for $U_\zeta(TB_r^+)$ now imply the following two lemmas:

\begin{lemma} \label{lemma:projectivecoverinjectivehull}
Let $\lambda \in X$.
\begin{enumerate}
\item In the category of rational $\Uztbr$-modules, $\whZr(\lambda)$ is the projective cover of $\lambda$ and the injective hull of $\lambda - 2(p^r\ell-1)\rho$.
\item In the category of rational $U_\zeta(TB_r^+)$-modules, $\whZr'(\lambda)$ is the projective cover of $\lambda - 2(p^r\ell-1)\rho$ and the injective hull of $\lambda$.
\end{enumerate}
\end{lemma}

Recall that, given a Hopf algebra $H$ and left $H$-modules $V$ and $W$, the space $\Hom_k(V,W)$ is made into a left $H$-module by the diagonal action
\begin{equation} \label{eq:Hopfdiagonalaction}
(h.f)(v) = \sum h^{(1)}.f(S(h^{(2)})v).
\end{equation}
In particular, the dual space $V^* = \Hom_k(V,k)$ is made into a left $H$-module.

\begin{lemma} \label{lemma:Zdual}
Let $\lambda \in X$. Then there exist left $\Uztgr$-module isomorphisms
\begin{align*}
\whZr(\lambda)^* &\cong Z_r(2(p^r\ell-1)\rho - \lambda) \qquad \text{and} \\
\whZr'(\lambda)^* &\cong Z_r'(2(p^r\ell-1)\rho - \lambda).
\end{align*}
\end{lemma}

\begin{proof}
The module $\whZr(\lambda)^*$ has highest weight $2(p^r\ell-1)\rho-\lambda$. Then there exists a $U_\zeta(TB_r^+)$-module homomorphism $2(p^r\ell-1)\rho-\lambda \rightarrow Z_r(\lambda)^*$, hence a $\Uztgr$-module homomorphism $\varphi: \whZr(2(p^r\ell-1)\rho-\lambda) \rightarrow \whZr(\lambda)^*$. Being the dual of an injective $\Uzbr$-module, $\whZr(\lambda)^*$ is projective for $\Uzbr$. It follows that $\whZr(\lambda)^*$ is isomorphic to the the projective cover of $2(p^r\ell-1)\rho-\lambda)$, hence that the map $\varphi$ must be surjective. Since the domain and range of $\varphi$ are each of the same finite dimension $(p^r\ell)^N$, $N = \abs{\Phi^+}$, the map $\varphi$ must be an isomorphism of $U_\zeta(G_r)$-modules. Similarly, the natural $U_\zeta(B_r)$-module homomorphism $Z_r'(\lambda)^* \rightarrow 2(p^r\ell-1)\rho-\lambda$ induces an isomorphism of $U_\zeta(G_r)$-modules $Z_r'(\lambda)^* \cong Z_r'(2(p^r\ell-1)\rho-\lambda)$.
\end{proof}

Let $N \in \set{B_r,G_r}$. We would like to characterize the rationally injective (resp.\ projective) $\Uztn$-modules in terms of their restriction to $\Uzn$. For this we utilize the Hopf algebra structure of $\Uztn$.

\begin{lemma} \label{lemma:projectivesareinjective}
Let $N \in \set{B_r,G_r}$. Any rationally projective $\Uztn$-module is rationally injective.
\end{lemma}

\begin{proof}
The proof is similar to the corresponding result for algebraic groups \cite[I.3.18]{Jantzen:2003}, but some care must be taken owing to the non-cocommutativity of the Hopf algebra $H:=\Uztn$. \emph{In this proof only, we assume that all $\Uztn$-modules are rational.}

Let $V$ be a finite-dimensional $H$-module, and let $P$ be a projective $H$-module. Then the tensor product $P \otimes V$ is also projective, because for any $H$-module $W$, we have $\Hom_H(P \otimes V,W) \cong \Hom_H(P,W \otimes V^*)$ by \cite[Proposition 1.18]{Andersen:1991}. Now, a short exact sequence $0 \rightarrow V_1 \rightarrow V_2 \rightarrow V_3 \rightarrow 0$ of finite-dimensional $H$-modules gives rise to the short exact sequence of $H$-modules
\[
0 \rightarrow P \otimes V_3^* \rightarrow P \otimes V_2^* \rightarrow P \otimes V_1^* \rightarrow 0.
\]
Each term is projective, so the short exact sequence splits, and the induced map on fixed points $(P \otimes V_2^*)^H \rightarrow (P \otimes V_1^*)^H$ is surjective. This map identifies with the natural map $\Hom_H(V_2,P) \rightarrow \Hom_H(V_1,P)$ by \cite[\S 3.10]{Jantzen:1996} and \cite[Lemma I.4.5]{Barnes:1985} (cf.\ also \cite[Proposition 2.9]{Andersen:1991}). Now let $0 \rightarrow V \rightarrow W$ be an arbitrary exact sequence of (rational) $H$-modules. Using Zorn's Lemma and the local finiteness of $W$, it follows that the natural map $\Hom_H(W,P) \rightarrow \Hom_H(V,P)$ is surjective, hence that $P$ is an injective $H$-module.
\end{proof}

The next result is similar to Lemma \ref{lemma:borelunipotentequivalence}.

\begin{lemma} \label{lemma:rationalinjectiveequivalence}
Let $N \in \set{B_r,G_r}$, and let $M$ be a finite-dimensional $\Uztn$-module. Then the following statements are equivalent:
\begin{enumerate}
\item $M$ is a rationally injective $\Uztn$-module.
\item $M$ is an injective $\Uzn$-module.
\item $M$ is a projective $\Uzn$-module.
\item $M$ is a rationally projective $\Uzn$-module.
\end{enumerate}
\end{lemma}

\begin{proof}
Statements (2) and (3) are equivalent because $\Uzn$ is a finite-dimensional Hopf algebra, while the implication (4) $\Rightarrow$ (1) is just Lemma \ref{lemma:projectivesareinjective}. The proof of the implication (3) $\Rightarrow$ (4) is essentially the same as the corresponding implication in Lemma \ref{lemma:borelunipotentequivalence}, replacing $\Uzur$ by $\Uzn$, $\Uzbr$ by $\Uztn$, $\Uztr$ by $\Uzo$, and replacing the index set in \eqref{eq:homweightdecomposition} by $p^r \ell \cdot X$. Finally, the implication (1) $\Rightarrow$ (2) follows by an argument similar to that used to prove \cite[Lemma 4.1(iii)]{Andersen:1992}; the details are left to the reader.
\end{proof}

%To prove the implication (1) $\Rightarrow$ (2), it suffices by a standard argument (see, e.g., \cite[Theorem 2.9.1]{Parshall:1991}) to show that induction from $\Uzn$ to $\Uztn$ is exact. This follows as in \cite[Corollary 2.3]{Andersen:1992a}, using the fact that the $p^r\ell$-th Steinberg module $\St_{p^r \ell} = L_\zeta((p^r\ell-1)\rho) \cong \whZr((p^r\ell-1)\rho) \cong \whZr'((p^r\ell-1)\rho)$ is injective for $\Uzgr$ (cf.\ \cite[\S 2.2]{Andersen:1992a}), hence also for $\Uzbr$ by \cite[Lemma I.4.3]{Barnes:1985}.

Having established the above characterization of rationally injective (resp.\ projective) modules for the algebras $\Uztbr$ and $\Uztgr$, the last two results of this section follow just as in the classical situation for algebraic groups, cf.\ \cite[\S II.11.1--11.4]{Jantzen:2003}. We leave the details to the reader.

\begin{lemma} \label{lemma:injectiveZtensor}
For all $\lambda,\mu \in X$, the $\Uztgr$-module $\whZr(\lambda) \otimes \whZr'(\mu)$ is rationally injective.
\end{lemma}

\begin{proposition} \label{proposition:injectivefiltration}
Let $M$ be a finite-dimensional rational $\Uztgr$-module.
\begin{enumalph}
\item $M$ is injective as a $\Uzbr$-module if and only if $M$ admits a filtration by $\Uztgr$-submodules with factors of the form $\whZr(\lambda)$, $\lambda \in X$.
\item $M$ is injective as a $U_\zeta(B_r^+)$-module if and only if $M$ admits a filtration by $\Uztgr$-submodules with factors of the form $\whZr'(\lambda)$, $\lambda \in X$.
\end{enumalph}
\end{proposition}

\subsection{Spectral sequences} \label{subsection:spectralsequences}

In Section \ref{subsection:Gsupportvarieties} we will study a certain spectral sequence \eqref{eq:uzgLHSspecseq}, which is usually constructed as the spectral sequence associated to the composite of two functors, cf. \cite[\S 3]{Ginzburg:1993}. Since the arguments in Section \ref{subsection:Gsupportvarieties} will require a product structure on \eqref{eq:uzgLHSspecseq} that is not apparent from the Grothendieck construction, as well as knowledge of the edge maps, we present here a construction of \eqref{eq:uzgLHSspecseq} that makes these features more apparent. Our construction mirrors \cite[Proposition 1.1]{Friedlander:1986b}.

Let $V$ be a rational $\Uz$-module. There exists a Lyndon--Hochschild--Serre spectral sequence
\begin{equation} \label{eq:uzbLHSspecseq}
{}'E_2^{i,j}(V) = \opH^i(\Uzb//\uzb,\opH^j(\uzb,V)) \Rightarrow \opH^{i+j}(\Uzb,V)
\end{equation}
computing the rational $\Uzb$-cohomology of the $\Uz$-module $V$. It can be constructed as in \cite[I.6.6]{Jantzen:2003}, replacing $k[G]$ there by $k[U_\zeta]$; cf.\ \cite[\S 1.34, 2.17]{Andersen:1991}. If $W$ is another rational $\Uz$-module, then there exists a morphism ${}'E_r(V) \otimes {}'E_r(W) \rightarrow {}'E_r(V \otimes W)$, which on the ${}'E_2$- and ${}'E_\infty$-pages is just the cup product for rational cohomology. The edge map
\[
\opH^j(\Uzb,V) \rightarrow {}'E_2^{0,j}(V) = \opH^j(\uzb,V)^{\Uzb}
\]
of \eqref{eq:uzbLHSspecseq} is just the restriction map.

Let $k[G] = \ind_1^G(k)$ be the coordinate ring of $G$, and let $k[G]^{(1)}$ denote $k[G]$ considered as a $\Uz$-module by pullback along the quantum Frobenius morphism $F_\zeta: \Uz \rightarrow \hy(G)$. Note that the map $a \otimes b \mapsto b \otimes a$ defines a $\Uz$-module isomorphism $k[G]^{(1)} \otimes V \cong V \otimes k[G]^{(1)}$. Also, the map $k[G] \otimes k[G] \cong k[G \times G] \rightarrow k[G]$ induced by the diagonal map $\Delta: G \rightarrow G \times G$ is a homomorphism of $G$-modules. Now, let $W$ be another rational $\Uz$-module. Then there exists a $\Uz$-module homomorphism
\[
\varphi: ( k[G]^{(1)} \otimes V ) \otimes ( k[G]^{(1)} \otimes W ) \cong (k[G] \otimes k[G])^{(1)} \otimes V \otimes W \rightarrow k[G]^{(1)} \otimes (V \otimes W).
\]
Set $E_r(V) = {}'E_r(k[G]^{(1)} \otimes V)$. Then $\varphi$ induces a morphism $E_r(V) \otimes E_r(W) \rightarrow E_r(V \otimes W)$. In particular, $E_r(k)$ is a spectral sequence of algebras, and $E_r(V)$ is a module over $E_r(k)$.

The triviality of $k[G]^{(1)}$ as a module for $\uzb$ implies that
\[
\opH^\bullet(\uzb,k[G]^{(1)} \otimes V) \cong k[G]^{(1)} \otimes \opH^\bullet(\uzb,V)
\]
as $\Uzb$-modules. Then $E_2^{i,j}(V)$ can be rewritten as
\begin{align*}
E_2^{i,j} & \cong \opH^i(\Uzb//\uzb,k[G]^{(1)} \otimes \opH^j(\uzb,V)) \\
& \cong \opH^i(B,k[G] \otimes \opH^j(\uzb,V)) \\
& \cong R^i \ind_B^G \opH^j(\uzb,V) & \text{by \cite[I.4.10]{Jantzen:2003}}.
\end{align*}
Next, note that $k[G]^{(1)} = (\ind_1^G(k))^{(1)} \cong \ind_{\uz}^{\Uz}(k)$ as $\Uz$-modules \cite[Theorem 2.3(ii)]{Andersen:1992}. Then
\begin{align*}
\opH^\bullet(\Uzb,k[G]^{(1)} \otimes V) & \cong \opH^\bullet(\Uzb,\ind_{\uz}^{\Uz}(k) \otimes V) \\
& \cong \opH^\bullet(\Uzb, \ind_{\uz}^{\Uz}(V)) & \text{by \cite[Proposition 2.16]{Andersen:1991}} \\
& \cong \opH^\bullet(\Uzg, \ind_{\uz}^{\Uz}(V)) & \text{cf.\ \cite[II.4.7]{Jantzen:2003}} \\
& \cong \opH^\bullet(\uzg,V) & \text{by the exactness of $\ind_{\uz}^{\Uz}(-).$}
\end{align*}
Thus, the spectral sequence $E_r(k[G]^{(1)} \otimes V)$ can be written as:
\begin{equation} \label{eq:uzgLHSspecseq}
E_2^{i,j}(V) = R^i \ind_B^G \opH^j(\uzb,V) \Rightarrow \opH^{i+j}(\uzg,V).
\end{equation}
Written in this form, the products $E_r(V) \otimes E_r(W) \rightarrow E_r(V \otimes W)$ are induced by the cup products for $\uzg$ and $\uzb$, and the edge map $\opH^j(\uzg,V) \rightarrow E_2^{0,j}(V) = \ind_B^G \opH^j(\uzb,V)$ is induced by Frobenius reciprocity from the restriction map $\opH^\bullet(\uzg,V) \rightarrow \opH^\bullet(\uzb,V)$.

\section{Injectivity criterion for modules with compatible torus action} \label{section:injectivitycriterion}

\subsection{} The main result of this section is the following theorem:

\begin{theorem} \label{theorem:injectivitycriterion}
Let $M$ be a finite-dimensional rational $U_\zeta(TG_r)$-module. Then $M$ is injective for $U_\zeta(G_r)$ if and only if the restriction $M|_{U_\zeta(U_{\alpha,r})}$ is injective for each $\alpha \in \Phi$.
\end{theorem}

One direction of the theorem is clear: The algebra $\Uzgr$ is flat (in fact, free) as a right module over each root subalgebra $U_\zeta(U_{\alpha,r})$ (apply the triangular decomposition for $U_\zeta$, and the explicit description of the PBW bases for $U_\zeta^+$ and $U_\zeta^-$). Then the injectivity of $M$ for $\Uzgr$ implies the injectivity of $M$ for each $U_\zeta(U_{\alpha,r})$ by \cite[Lemma I.4.3]{Barnes:1985}. This direction of the theorem does not require a compatible $\Uzo$-structure on $M$.

To prove the other direction of the theorem, in Section \ref{subsection:reductionBorel} we reduce the problem to the case of a rational $\Uztgr$-module that is injective over the Borel subalgebras $\Uzbr$ and $U_\zeta(B_r^+)$. Then in Section \ref{subsection:injectiveBorels} we prove that such a module is injective for $\Uzbr$ (resp.\ $U_\zeta(B_r^+)$) if and only if its restriction to each root subalgebra is injective. Our overall strategy is the same as that in \cite{Cline:1985} for the classical algebraic group situation, but extra care must be taken, especially in Section \ref{subsection:injectiveBorels}, owing to the non-cocommutativity of the Hopf algebras under consideration and the complicated relations between root vectors in $\Uz$.

\subsection{Reduction to Borel subalgebras} \label{subsection:reductionBorel}

\begin{theorem} \label{theorem:reductionborel}
Let $M$ be a finite-dimensional rational $\Uztgr$-module. Then $M$ is injective for $U_\zeta(G_r)$ if and only if the restrictions $M|_{U_\zeta(B_r)}$ and $M|_{U_\zeta(B_r^+)}$ are injective.
\end{theorem}

\begin{proof}
Suppose that $M$ is injective for $\Uzgr$. Then $M$ is injective for $\Uzbr$ and $U_\zeta(B_r^+)$ by \cite[Lemma I.4.3]{Barnes:1985}, because $\Uzgr$ is free as a right module for either $\Uzbr$ or $U_\zeta(B_r^+)$. Now suppose that $M$ is injective as a module for $\Uzbr$ and as a module for $U_\zeta(B_r^+)$. Then by Lemma \ref{lemma:Zdual} and Proposition \ref{proposition:injectivefiltration}, the $\Uztgr$-module $\End_k(M) \cong M \otimes M^*$ admits a filtration by $\Uztgr$-submodules with factors of the form $\whZr(\lambda) \otimes \whZr'(\mu)$ ($\lambda,\mu \in X$). These factors are injective for $\Uzgr$ by Lemma \ref{lemma:injectiveZtensor}, so $\End_k(M)$ must be an injective (equivalently, projective) $\Uzgr$-module.

Now, since $\End_k(M)$ is projective for $\Uzgr$, then so is $\End_k(M) \otimes M$, cf.\ the proof of Lemma \ref{lemma:projectivesareinjective}. The $\Uzgr$-module embedding $M \cong (k \cdot \id) \otimes M \hookrightarrow \End_k(M) \otimes M$ splits via the $\Uzgr$-module homomorphism $\End_k(M) \otimes M \rightarrow M$, $\varphi \otimes m \mapsto \varphi(m)$, so $M$ is isomorphic to a $\Uzgr$-direct summand of $\End_k(M) \otimes M$. Then $M$ is projective (equivalently, injective) as a $\Uzgr$-module.
\end{proof}

\subsection{Injectivity for Borel subalgebras} \label{subsection:injectiveBorels}

Since we can twist the structure map of any $U_\zeta(B_r^+)$-module $M$ by the automorphism $\omega$ to make it a $U_\zeta(B_r)$-module, to complete the proof of Theorem \ref{theorem:injectivitycriterion} it now suffices to prove the following theorem:

\begin{theorem} \label{theorem:injectiveBorel}
Let $M$ be a rational $\Uztbr$-module. Then $M$ is injective for $\Uzbr$ if and only if $M$ is injective for each root subalgebra $U_\zeta(U_{\alpha,r})$, $\alpha \in \Phi^+$.
\end{theorem}

We will actually show that $M$ is injective for $\Uzur$, but this is equivalent to injectivity for $\Uzbr$ by Lemma \ref{lemma:borelunipotentequivalence}. The proof is by induction, using the algebras $A_m := \Uzurm$ ($1 \leq m \leq N$) defined in Section \ref{subsection:FLkernels}. The key to the induction argument is the fact that the convex ordering $\set{\gamma_1,\ldots,\gamma_N}$ on $\Phi^+$ defined in Section \ref{subsection:QEAs} is compatible with a sequence of total orderings $\leqslant_m$ on the vector space $\R \Phi$.

\begin{lemma} \label{lemma:rootorder} \textup{(cf.\ \cite[Lemma 3.1]{Cline:1985})}
For each $0 \leq m < N$, there exists an ordering $\leqslant_m$ on the Euclidean space $\R \Phi$ such that $\gamma_i >_m 0$ for $i \leq m$, and $\gamma_{i} <_m 0$ for $i > m$.
\end{lemma}

\begin{proof}
The proof goes by induction on $m$. To start, choose $\leqslant_0$ to be any total ordering on $\R\Phi$ such that $\gamma <_0 0$ for all $\gamma \in \Phi^+$. Now let $m \geq 0$, and assume by way of induction that there exists a total ordering $\leqslant_m$ on $\R\Phi$ satisfying the conditions of the lemma. Let $\Phi_m^+$ (resp.\ $\Phi_m^-$) denote the positive (resp.\ negative) system of roots determined by $\leqslant_m$. Then $\gamma_{m+1} \in \Phi_m^-$. We claim that $\gamma_{m+1}$ is simple with respect to the ordering $\leqslant_m$. Indeed, suppose $\gamma_{m+1} = \alpha_1 + \alpha_2$ for some $\alpha_1,\alpha_2 \in \Phi_m^-$. There are three cases to consider:
\begin{enumerate}
\item $\alpha_1 = \gamma_i \in \Phi_m^-$, $\alpha_2 = \gamma_j \in \Phi_m^-$, with $(m+1) < i < j \leq N$. Then $\gamma_{m+1} = \gamma_i + \gamma_j$, an impossibility because $\set{\gamma_1,\ldots,\gamma_N}$ is a convex ordering of $\Phi^+$, and $(m+1) \notin [i,j]$.
\item $\alpha_1 = -\gamma_i \in \Phi_m^-$, $\alpha_2 = -\gamma_j \in \Phi_m^-$, with $1 \leq i < j \leq m$. Then $-\gamma_{m+1} = \gamma_i + \gamma_j$, an impossibility because the sum of two roots in $\Phi^+$ is never a root in $\Phi^-$.
\item $\alpha_1 = -\gamma_i \in \Phi_m^-$, $\alpha_2 = \gamma_j \in \Phi_m^-$, with $1 \leq i < (m+1) < j \leq N$. Then $\gamma_i + \gamma_{m+1} = \gamma_j$, an impossibility because $\set{\gamma_1,\ldots,\gamma_N}$ is a convex ordering of $\Phi^+$, and $j \notin [i,m+1]$.
\end{enumerate}
So $\gamma_{m+1}$ is simple with respect to the ordering $\leqslant_m$.

Now, the Weyl group $W$ acts transitively on the collection of positive systems in $\Phi$, and to each positive system $S$ of roots in $\Phi$, we can associate a total ordering $\preceq$ on $\R\Phi$ such that the positive roots in $\Phi$ with respect to $\preceq$ are precisely those in $S$ (e.g., enumerate a set of simple roots in $S$, hence an ordered basis for $\R\Phi$, and then take $\preceq$ to be the standard lexicographic ordering on $\R\Phi$ with respect to that basis). Let $\leqslant_{m+1}$ be a total ordering on $\R\Phi$ with associated positive system $\Phi_{m+1}^+:=s_{\gamma_{m+1}}(\Phi_m^+)$. Since $\gamma_{m+1}$ is simple with respect to $\leqslant_m$, it follows that $\Phi_{m+1}^+ \cap \Phi^+ = \set{\gamma_1,\ldots,\gamma_m,\gamma_{m+1}}$. Then $\leqslant_{m+1}$ satisfies the conditions of the lemma. This completes the induction step, and the proof of the lemma.
\end{proof}

\begin{proof}[Proof of Theorem \ref{theorem:injectiveBorel}]
Let $M$ be an injective $\Uzur$-module. Then, since for each $\alpha \in \Phi^+$ the algebra $\Uzur$ is free as a right module over $U_\zeta(U_{\alpha,r})$, the restriction $M|_{U_\zeta(U_{\alpha,r})}$ is injective for each positive root $\alpha \in \Phi^+$ by \cite[Lemma I.4.3]{Barnes:1985}. (This direction of the theorem does not require the action of $\Uzo$.) Now let $M$ be a finite-dimensional rational $\Uztbr$-module for which the restricted modules $M|_{U_\zeta(U_{\alpha,r})}$ are all injective (equivalently, free). We prove, for each $1 \leq m \leq N$, that $M$ is free over $A_m = U_\zeta(U_{r,m})$. Taking $m=N$ then yields the desired result.

By assumption, $M$ is free as a module over $A_1 = U_\zeta(U_{\gamma_1,r})$, so let $m \geq 1$ and assume by induction that $M$ is free as a module over $A_m$. We show that $M$ is free over $A_{m+1}$. By Lemma \ref{lemma:weightbasis}, $M$ admits an $A_m$-basis $S = \set{s_1,\ldots,s_r}$ consisting of weight vectors for $U_\zeta^0$. Set $\alpha = \gamma_{m+1}$. By Lemma \ref{lemma:rootorder}, there exists an ordering $\preceq$ on $\R\Phi$ with $\alpha \prec 0$, and $\gamma_i \succ 0$ if $i \leq m$. Choose $v \in S$ of weight $\lambda$ minimal with respect to $\preceq$. Then $\lambda$ is minimal with respect to $\preceq$ among all the weights for $\Uzo$ in $M$. Since $\alpha \prec 0$, the minimality of $\lambda$ implies that $F_\alpha^{(n)} v = 0$ for all $1 \leq n < p^r\ell$. Set $\int_\alpha = F_\alpha^{(p^r\ell-1)}$. By the freeness of $M$ over $U_\zeta(U_{\alpha,r})$, there exists $w \in M$ with $\int_\alpha w = v$. Now
\begin{equation*}
\textstyle \int_{m+1} = F_{\gamma_1}^{(p^r\ell-1)} \cdots F_{\gamma_m}^{(p^r\ell-1)} F_{\gamma_{m+1}}^{(p^r\ell-1)} = \int_m \int_\alpha,
\end{equation*}
so $\int_{m+1} w = \int_m \int_\alpha w = \int_m v$, and $\int_m v \neq 0$ because $v$ is an $A_m$-basis vector for $M$. Then $\int_{m+1} w \neq 0$.

Now write $w = \sum_{\mu \in X} w_\mu$, the decomposition of $w$ as a finite sum of $\Uzo$-weight vectors. Since $\int_{m+1} w \neq 0$, there exists $\mu \in X$ with $\int_{m+1} w_\mu \neq 0$. It follows that $M' := A_{m+1} . w_\mu$ is free as a module for $A_{m+1}$. By Lemma \ref{lemma:indecomposableinjectives}, $M'$ is injective in the category of rational $TA_{m+1}$-modules, so we can decompose $M$ as a direct sum $M = M' \oplus M''$ for some $TA_{m+1}$-submodule $M''$ of $M$. By induction on dimension, $M''$ is free over $A_{m+1}$. Then $M = M' \oplus M''$ is free for $A_{m+1}$.
\end{proof}

\section{Support varieties for Borel subalgebras} \label{section:Bsupportvarieties}

\subsection{}

Let $H$ be a Hopf algebra, and let $M$ be a left $H$-module. Set $\opH(H,k) = \opH^{2\bullet}(H,k)$, the subring of the cohomology ring $\opH^\bullet(H,k)$ generated by elements of even degree. Then $\opH(H,k)$ is a commutative ring under the cup product \cite[Corollary 4.3]{Mac-Lane:1995}. Suppose that $\opH(H,k)$ is finitely-generated as an algebra. Then $\calV_H(k) := \Maxspec(\opH(H,k))$, the maximal ideal spectrum of $\opH(H,k)$, is an affine variety. Define $J_H(M)$ to be the annihilator for the cup product action of $\opH(H,k)$ on $\Ext_H^\bullet(M,M) \cong \Ext_H^\bullet(k,\Hom_k(M,M)) \cong \Ext_H^\bullet(k,M \otimes M^*)$. Equivalently, $J_H(M)$ is the kernel of the graded algebra homomorphism $\opH(H,k) \rightarrow \opH^{2\bullet}(H,\Hom_k(M,M))$ induced by the natural map $k \rightarrow \Hom_k(M,M)$, $1 \mapsto \id$. Then the support variety $\calV_H(M)$ is defined to be the conical subvariety of $\calV_H(k)$ determined by the ideal $J_H(M)$.

We now turn our attention to studying cohomological support varieties for the Borel subalgebras $\uzb$ and $\uzbp$ of the small quantum group $\uzg$. The automorphism $\omega$ of $\Uz$ restricts to an isomorphism $\uzb \stackrel{\sim}{\rightarrow} \uzbp$, so any results we prove for $\uzbp$ can be immediately translated into results for $\uzb$. Thus,  in order to simplify some calculations, in this section we choose to work exclusively with the positive Borel subalgebra $\uzbp$.

In what follows we make the following assumptions:

\begin{assumption} \label{assumption:ellandp}
Let $k$ be an algebraically closed field. Assume $p = \chr(k)$ to be odd or zero, and to be good for $\Phi$. (For the definition of a good prime, see, e.g., \cite[II.4.22]{Jantzen:2003}.) Assume $\zeta \in k$ to be a primitive $\ell$-th root of unity, with $\ell \in \N$ odd, coprime to three if $\Phi$ has type $G_2$, and $\ell \geq h$, $h$ the Coxeter number of $\Phi$.
\end{assumption}

Under these assumptions, we have the following descriptions for $\opH(\uzg,k)$ and $\opH(\uzbp,k)$:

\begin{theorem} \label{theorem:Bisomorphism} \textup{(cf.\ \cite[Lemma 2.6]{Ginzburg:1993} and \cite[Corollary 4.20]{Drupieski:2009})} We have $\opH^{\textup{odd}}(\uzbp,k) = 0$. There exists a natural $B^+$-algebra isomorphism $\opH(\uzbp,k) \cong S(\fraku^{+*})$, $S(\fraku^{+*})$ the symmetric algebra on the dual space $(\fraku^+)^*$, where the action of $B^+$ on $\opH(\uzbp,k)$ is induced by the adjoint action of $\Uzbp$ on $\opH^\bullet(\uzbp,k)$, and the action of $B^+$ on $S(\fraku^{+*})$ is induced by the coadjoint action of $B^+$ on $\fraku^{+*}$.
\end{theorem}

\begin{theorem} \label{theorem:Gisomorphism} \textup{(cf.\ \cite[Theorem 3]{Ginzburg:1993} and \cite[Corollary 4.23]{Drupieski:2009})} We have $\opH^{\textup{odd}}(\uzg,k) = 0$. There exists a natural graded $G$-algebra isomorphism $\opH(\uzg,k) \cong k[\calN]$, $k[\calN]$ the coordinate ring of the nullcone $\calN \subset \g$, where the action of $G$ on $\opH(\uzg,k)$ is induced by the adjoint action of $\Uz$ on $\opH^\bullet(\uzg,k)$, and the action of $G$ on $k[\calN]$ is induced by the usual adjoint action of $G$ on $\calN \subset \g$.
\end{theorem}

Theorem \ref{theorem:Bisomorphism} implies that for any $\uzbp$-module $M$, the support variety $\Vbp(M)$ identifies with a conical subvariety of $\fraku^+ = \Lie(U^+)$, and if $M$ lifts to a $\Uzbp$-module, then $\Vbp(M)$ is even a $B^+$-stable subvariety of $\fraku^+$. Similarly, Theorem \ref{theorem:Gisomorphism} implies that for any $\uzg$-module (resp.\ $\Uz$-module) $M$, the support variety $\Vg(M)$ identifies with a conical ($G$-stable) subvariety of the nullcone $\calN$.

In this section we generalize the results of \cite{Friedlander:1987} to show that the root vector $e_\alpha \in \fraku^+$ is contained in the support variety $\Vbp(M)$ of the $\uzbp$-module $M$ if and only if $M$ is not projective (equivalently, injective) for the root subalgebra  $u_\zeta(e_\alpha)$. The proof is inductive in nature, and utilizes the Lyndon--Hochschild--Serre spectral sequence \eqref{eq:LHSspecseq} in order to show that the action of the coordinate function $x_\alpha \in S(\fraku^{+*}) \cong \opH(\uzbp,k)$ on certain cohomology groups is nilpotent. Since most of the algebras under consideration in the induction argument are not Hopf subalgebras of $\uzbp$ but are only one-sided coideal subalgebras, we must perform a number of technical calculations that would otherwise be unnecessary when dealing with Hopf algebras. Most of these technical calculations are contained in Section \ref{subsection:injectivityrootsubalgebras}.

\subsection{Cohomology products} \label{subsection:cohomologyproducts}

Let $\Lambda$ and $\Lambda'$ be algebras over $k$. Let $V$ and $W$ be left $\Lambda$-modules, and let $V'$ and $W'$ be left $\Lambda'$-modules. Set $\Omega = \Lambda \otimes \Lambda'$. Recall that the \emph{wedge product} is a family of $k$-bilinear maps
\begin{equation} \label{eq:wedgeproduct}
\vee: \Ext_\Lambda^n(V,W) \otimes \Ext_{\Lambda'}^m(V',W') \rightarrow \Ext_\Omega^{n+m}(V \otimes V',W \otimes W').
\end{equation}
It is defined as follows: Take projective resolutions $X \rightarrow V$ and $X' \rightarrow V'$ by $\Lambda$- and $\Lambda'$-modules, respectively. Then, for each $n,m \in \N$, $X_n \otimes X_m'$ is projective for $\Omega$, and by the K\"{u}nneth Theorem, $X \otimes X'$ is an $\Omega$-projective resolution of $V \otimes V'$. Given $f \in \Hom_\Lambda(X,W)$ and $g \in \Hom_{\Lambda'}(X',W')$, define $f \vee g \in \Hom_\Omega(X \otimes X', W \otimes W')$ by $(f \vee g)(x \otimes x') = f(x) \otimes g(x')$. Then \eqref{eq:wedgeproduct} is the map in cohomology induced by $\vee: \Hom_\Lambda(X,W) \otimes \Hom_{\Lambda'}(X',W') \rightarrow \Hom_\Omega(X \otimes X',W \otimes W')$.

Suppose $\Lambda \stackrel{\varepsilon}{\rightarrow} k$ and $\Lambda' \stackrel{\varepsilon}{\rightarrow} k$ are augmented algebras over $k$, and that $V = V' = k$. Then we could take $X = \upB(\Lambda)$ and $X' = \upB(\Lambda')$, the left bar resolutions for $\Lambda$ and $\Lambda'$, respectively. Since $\upB(\Lambda) \otimes \upB(\Lambda')$ and $\upB(\Omega) = \upB(\Lambda \otimes \Lambda')$ are both $\Omega$-projective resolutions of $k \cong k \otimes k$, there exists an $\Omega$-module chain map $\varphi: \upB(\Omega) \rightarrow \upB(\Lambda) \otimes \upB(\Lambda')$, unique up to homotopy, lifting the identity $k \rightarrow k$. An explicit choice for $\varphi$ is given by the following formula (cf.\ \cite[XI.7]{Cartan:1999}):
\[ \varphi([\lambda_1 \otimes \lambda_1',\ldots,\lambda_n \otimes \lambda_n']) = \sum_{i=0}^n [\lambda_1,\ldots,\lambda_i]\varepsilon(a_{i+1} \cdots a_n) \otimes \lambda_1' \cdots \lambda_i'[\lambda_{i+1}',\ldots,\lambda_n']. \]

Now let $H$ be a Hopf algebra with comultiplication $\Delta: H \rightarrow H \otimes H$, and let $V$ and $W$ be left $H$-modules. Then the cup product
\begin{equation} \label{eq:Hcupproduct}
\cup: \opH^n(H,V) \otimes \opH^m(H,W) \rightarrow \opH^{n+m}(H,V \otimes W)
\end{equation}
is the composite of the wedge product $\vee: \Ext_H^n(k,V) \otimes \Ext_H^m(k,W) \rightarrow \Ext_{H \otimes H}^{n+m}(k, V \otimes W)$ with the map $\Ext_{H \otimes H}^\bullet(k,V \otimes W) \rightarrow \Ext_H^\bullet(k,V \otimes W)$ induced by $\Delta$. If $\zeta \in \opH^n(H,V)$ and $\eta \in \opH^m(H,W)$ are represented by cocycles $f \in \Hom_H(\upB_n(H),V)$ and $g \in \Hom_H(\upB_m(H),W)$, then the cup product $\zeta \cup \eta$ is the cohomology class of the cocycle $f \cup g:= f \vee g \circ \varphi \circ \Delta \in \Hom_H(\upB_{n+m}(H),V \otimes W)$. Explicitly, writing $\Delta(h_i) = h_i^{(1)} \otimes h_i^{(2)}$ and using the fact that $(\varepsilon \otimes 1) \circ \Delta = \id$,
\begin{equation} \label{eq:cupproductcocycle}
(f \cup g)([h_1,\ldots,h_{n+m}]) = f([h_1^{(1)},\ldots,h_n^{(1)}]) \otimes h_1^{(2)} \cdots h_n^{(2)}g([h_{n+1},\ldots,h_{n+m}]),
\end{equation}
In particular, if $W$ has trivial $H$-action, then \eqref{eq:cupproductcocycle} reduces to
\begin{equation} \label{eq:Yonedaformula}
(f \cup g)([h_1,\ldots,h_{n+m}]) = f([h_1,\ldots,h_n]) \otimes g([h_{n+1},\ldots,h_{n+m}]).
\end{equation}

Suppose $A \subseteq H$ is a left coideal subalgebra, that is, $A$ is a subalgebra of $H$ and $\Delta(A) \subseteq H \otimes A$. Then, given an $A$-module $W$, we have the composite
\begin{equation}
\Delta^* \circ \vee : \opH^n(H,k) \otimes \opH^m(A,W) \rightarrow \Ext_{H \otimes A}^{n+m}(k,k \otimes W) \rightarrow \opH^{n+m}(A,W).
\end{equation}
We call this the cup product action of $\opH^\bullet(H,k)$ on $\opH^\bullet(A,W)$. By abuse of notation we also denote it by the symbol $\cup$. Then $\cup$ admits a description at the level of chain complexes by exactly the same formula as \eqref{eq:cupproductcocycle}, interpreting the $h_i$ now as elements of $A$. By \cite[Theorem VIII.4.1]{Mac-Lane:1995}, if $\zeta \in \opH^n(H,k)$ and $\eta \in \opH^m(A,W)$, then $\zeta \cup \eta = (-1)^{nm} \eta \circ \res_A^H(\zeta)$, where $\res_A^H(\zeta)$ denotes the image of $\zeta$ under the cohomological restriction map $\opH^\bullet(H,k) \rightarrow \opH^\bullet(A,k)$, and where
\begin{equation} \label{eq:Yonedacomposition}
\circ: \Ext_A^m(k,W) \otimes \Ext_A^n(k,k) \rightarrow \Ext_A^{n+m}(k,W)
\end{equation}
denotes the Yoneda composition of extensions. Of course, the Yoneda product \eqref{eq:Yonedacomposition} depends only on the ring structure of $A$ (and the $A$-module structure of $W$), so the cup product action of $\opH^\bullet(H,k)$ on $\opH^\bullet(A,W)$ is independent of the particular comultiplication map for $H$.

\subsection{\texorpdfstring{Left coideal subalgebras in $U_q^+$}{Left coideal subalgebras in U+}}

Let $w \in W$, and suppose that the reduced expression for $w_0$ chosen in \S \ref{subsection:QEAs} begins with a reduced expression for $w$, that is, $w_0 = w w'$ with $\ell(w_0) = \ell(w) + \ell(w')$. Here $\ell: W \rightarrow \N$ denotes the usual length function on $W$. Set $m = \ell(w)$. Then the algebra $U_q^+[w]$ defined in \cite[\S 8.24]{Jantzen:1996} is the subalgebra of $U_q^+$ generated by the root vectors $E_{\gamma_1},E_{\gamma_2},\ldots,E_{\gamma_m}$. Heckenberger and Schneider \cite{Heckenberger:2009} have shown that every right coideal subalgebra in $U_q(B^+) = U_q^0 U_q^+$ containing $U_q^0$ has the form $U_q^0 U_q^+[w]$ for some $w \in W$. (Right coideals are defined by the condition $\Delta(A) \subseteq A \otimes H$.) To each $w \in W$ we can also associate a left coideal subalgebra of $U_q(B^+)$, namely, the algebra $T_w(U_q^+[w'])$.

\begin{lemma} \label{lemma:leftcoidealsubalgebras}
Let $w \in W$. Write $w_0 = ww'$ with $\ell(w_0) = \ell(w) + \ell(w')$. Then $T_w(U_q^+[w'])$ is a left coideal subalgebra of $U_q(B^+)$.
\end{lemma}

\begin{proof}
Suppose that the reduced expression for $w_0$ chosen in \S \ref{subsection:QEAs} begins with a reduced expression for $w$. Set $m = \ell(w)$. Then $T_w(U_q^+[w'])$ is the subalgebra of $U_q^+$ generated by the root vectors $E_{\gamma_{m+1}},E_{\gamma_{m+2}},\ldots,E_{\gamma_N}$. Since $T_w(U_q^+[w']) \subset U_q^+$, we know that $\Delta(T_w(U_q^+[w'])) \subseteq U_q(B^+) \otimes U_q^+$. Since $T_w(U_q^+[w']) = \set{u \in U_q^+: T_w^{-1}(u) \in U_q^+}$, to prove the lemma it suffices to show that
\begin{equation*} \label{eq:coidealinverse}
(1 \otimes T_w^{-1}) \circ \Delta(T_w(U_q^+[w'])) \subseteq U_q(B^+) \otimes U_q^+.
\end{equation*}
This last claim follows from \cite[Proposition C.5(2)]{Andersen:1994}.
\end{proof}

\begin{corollary} \label{corollary:leftcoideals}
Let $w_0 = s_{\beta_1} \cdots s_{\beta_N}$ be an arbitrary reduced expression for $w_0 \in W$, and let $\set{\gamma_1,\ldots,\gamma_N}$ be the corresponding convex ordreing of $\Phi^+$. Let $E_{\gamma_m} \in U_q^+$ be the positive root vector of weight $\gamma_m$ as defined in \S \ref{subsection:QEAs}. Then $\Delta(E_{\gamma_m}) \in V_m \otimes W_m$, where $V_m \subset U_q(B^+)$ is the subalgebra generated by $U_q^0 \cup \set{E_{\gamma_1},\ldots,E_{\gamma_m}}$, and $W_m \subset U_q^+$ is the subalgebra generated by $\set{E_{\gamma_m},\ldots,E_{\gamma_N}}$.
\end{corollary}

\begin{proof}
Set $w = s_{\beta_1} \cdots s_{\beta_m}$, $w' = s_{\beta_1} \cdots s_{\beta_{m-1}}$, and $w'' = s_{\beta_m} \cdots s_{\beta_N}$, so that $w_0 = w' w''$. Now use the fact that $E_{\gamma_m} \in U_q^+[w] \cap T_{w'}(U_q^+[w''])$, $U_q^0 U_q^+[w]$ is a right coideal subalgebra of $U_q(B^+)$, and $T_{w'}(U_q^+[w''])$ is a left coideal subalgebra of $U_q^+$.
\end{proof}

Let $\tau$ be the anti-automorphism of $U_q$ defined by $\tau(E_\alpha) = E_\alpha$, $\tau(F_\alpha) = F_\alpha$, and $\tau(K_\alpha) = K_\alpha^{-1}$, $\alpha \in \Pi$. We can twist the Hopf algebra structure maps $(\Delta,S,\varepsilon)$ for $U_q$ by $\tau$ to obtain a new set of Hopf algebra structure maps $(\Deltatau,\Stau,\etau) = ((\tau \otimes \tau) \circ \Delta \circ \tau, \tau \circ S^{-1} \circ \tau, \varepsilon)$ for $U_q$. Considering this new Hopf algebra structure on $U_q$, we get:

\begin{lemma} \label{lemma:twistedleftcoideal}
Let $w \in W$. Write $w_0 = ww'$ with $\ell(w_0) = \ell(w) + \ell(w')$. Then $U_q^+[w]$ is a left coideal subalgebra and $U_q^0 T_w(U_q^+[w'])$ is a right coideal subalgebra for the twisted Hopf algebra structure $(\Deltatau,\Stau,\etau)$ on $U_q(B^+)$.
\end{lemma}

\begin{proof}[Proof sketch.]
To prove the statement for $U_q^+[w]$, use \cite[Proposition C.4]{Andersen:1994} and the identity $\tau \circ T_\alpha \circ \tau = T_\alpha^{-1}$, $\alpha \in \Pi$. To prove the statement for $U_q^0 T_w(U_q^+[w'])$, imitate the proof of Lemma \ref{lemma:leftcoidealsubalgebras}.
\end{proof}

\begin{corollary} \label{corollary:twistedcoideals}
Let $w_0 = s_{\beta_1} \cdots s_{\beta_N}$ be an arbitrary reduced expression for $w_0 \in W$, and let $\set{\gamma_1,\ldots,\gamma_N}$ be the corresponding convex ordreing of $\Phi^+$. Let $E_{\gamma_m} \in U_q^+$ be the positive root vector of weight $\gamma_m$ as defined in \S \ref{subsection:QEAs}. Then $\Deltatau(E_{\gamma_m}) \in V_m' \otimes W_m'$, where $V_m' \subset U_q(B^+)$ is the subalgebra generated by $U_q^0 \cup \set{E_{\gamma_m},\ldots,E_{\gamma_N}}$, and $W_m' \subset U_q^+$ is the subalgebra generated by $\set{E_{\gamma_1},\ldots,E_{\gamma_m}}$.
\end{corollary}

\subsection{The Lyndon--Hochschild--Serre spectral sequence} \label{subsection:LHSspecseq}

Let $A$ be an augmented algebra, and let $B$ be a normal subalgebra of $A$. Write $B_\varepsilon$ for the augmentation ideal of $B$, and set $K = AB_\varepsilon$, the two-sided ideal in $A$ generated by $B_\varepsilon$. Then $A//B = A/K$. Assume that $A$ is flat as a right $B$-module. Then for any $A$-module $V$, there exists a unique natural action of $A$ on the cohomology groups $\opH^\bullet(B,V)$ extending the action of $A$ on the space of invariants $V^B$. Moreover, there exists a spectral sequence satisfying
\begin{equation} \label{eq:LHSspecseq}
E_1^{i,j} = E_1^{i,j}(V) = \Hom_{A//B}(\upB_i(A//B),\opH^j(B,V)) \Rightarrow \opH^{i+j}(A,V).
\end{equation}
This is the Lyndon--Hochschild--Serre (LHS) spectral sequence associated to the algebra extension $0 \rightarrow B \rightarrow A \rightarrow A//B \rightarrow 0$ and the $A$-module $V$. Depending on the existence of additional structure on $A$ and $B$, the LHS spectral sequence can be constructed in several equivalent ways, cf.\ \cite[Chapter VIII]{Barnes:1985}. We are interested in a construction due to Hochschild and Serre, described in \cite[Chapter IV]{Barnes:1985}. For this construction we make the additional assumption that $A$ is projective as a right $B$-module. The main points of the construction are summarized below.

Set $C =C^\bullet(A,V) := \Hom_A(\upB_\bullet(A),V)$. Write $\delta$ for the differential on $C$. Define a decreasing filtration $F$ on $C$ by setting $F^0 C^n = C^n$, $F^{n+1} C^n = 0$, and for $0 < p \leq n$,
\[ F^p C^n = \set{f \in C^n(A,V): f([a_1,\ldots,a_n]) = 0 \text{ if any of $a_{n-p+1},\ldots,a_n$ is in $K$}}. \]
(In \cite{Barnes:1985}, this filtration is denoted by $F*$.) Then $F$ makes $C$ a filtered differential graded module, and \eqref{eq:LHSspecseq} is the associated spectral sequence. The identification
\begin{equation} \label{eq:eta}
\eta: E_1^{p,n-p} \stackrel{\sim}{\rightarrow} \Hom_{A//B}(\upB_p(A//B),\opH^{n-p}(B,V))
\end{equation}
is made as follows: Let $f \in F^pC^n$ be such that $\delta(f) \in F^{p+1}C^{n+1}$. Then $f$ represents a relative cocycle $[f] \in E_0^{p,n-p} = F^pC^n/F^{p+1}C^n$. Given $x_1,\ldots,x_p \in A//B$, define $\eta([f])([x_1,\ldots,x_p]) = [\varphi]$, the cohomology class of $\varphi \in \Hom_B(\upB_{n-p}(B),V)$, where $\varphi$ is defined by
\begin{equation} \label{eq:representativecocycle}
\varphi([b_1,\ldots,b_{n-p}]) := f([b_1,\ldots,b_{n-p},x_1,\ldots,x_p]).
\end{equation}
That $\varphi$ is a well-defined cocycle follows from the fact that $f$ represents a relative cocycle in $E_0^{p,n-p}$. That $\eta$ defines an isomorphism follows from \cite[Theorem III.1.5 and Lemma IV.3.1]{Barnes:1985}.

\begin{remark}
Suppose $A$ is a Hopf algebra, and that $B$ is a Hopf subalgebra of $A$. Then the LHS spectral sequence $E_r(k)$ is naturally a spectral sequence of algebras, with products induced by a natural algebra structure on $C(A,k)$, and $E_r(V)$ is naturally a module over $E_r(k)$, with module structure induced by a natural $C(A,k)$-action on $C(A,V)$; see \cite[Theorem IV.3.6]{Barnes:1985}. The filtration $F$ defined here is not compatible with these algebra and module structures, and this is part of the reason for the technical calculations we must conduct in Section \ref{subsection:injectivityrootsubalgebras} (the other reason for the calculations being that we are not dealing with Hopf algebras in Section \ref{subsection:injectivityrootsubalgebras}). Still, the filtration $F$ defined here seems to be the most useful for the purposes of our induction argument.
\end{remark}

\subsection{Injectivity for root subalgebras} \label{subsection:injectivityrootsubalgebras}

Let $\set{x_\alpha: \alpha \in \Phi^+} \subset \fraku^{+*}$ be the dual basis corresponding to the root vector basis $\set{e_\alpha: \alpha \in \Phi^+}$ for $\fraku^+$. Then $\opH(\uzbp,k) \cong S(\fraku^{+*})$ is the polynomial algebra on the degree two generators $x_\alpha$, $\alpha \in \Phi^+$. Before we state the main result of this section, we collect some information on the polynomial generators for $\opH(\uzbp,k)$.

\subsubsection{} \label{subsubsection:one} Let $f \in \Hom_{\uzbp}(\upB_2(\uzbp),k)$ be a cocycle representative for $x_\alpha$. Since $x_\alpha$ is a weight vector of weight $-\ell\alpha$ for $\Uzo$, we may assume that $f$ is a weight vector of weight $-\ell\alpha$ for the adjoint action of $\Uzo$ on $C^2(\uzbp,k)$. Suppose $\alpha$ is a simple root. Then $f$ has support in the subspace of $\upB_2(\uzbp)$ spanned by all vectors of the form $[x_1,x_2]$ with $x_1,x_2 \in \uzo u_\zeta(e_\alpha)$.

\subsubsection{} \label{subsubsection:two} Write $\Phi^+ = \set{\gamma_1,\ldots,\gamma_N}$ with $\gamma_i = w_i(\beta_i)$ as in \S \ref{subsection:QEAs}. Suppose that $\alpha = \gamma_m$. Set $\beta = \beta_m$, and set $w = w_m$. We have $\opH(\uzg,k) \cong k[\calN]$, and the restriction map $\opH(\uzg,k) \rightarrow \opH(\uzbp,k) \cong S(\fraku^{+*})$ is just the restriction of functions. For each $w \in W$, the braid group operator $T_w$ induces an automorphism of $\uzg$, hence an automorphism $(T_w^{-1})^*$ of $\opH(\uzg,k)$. This automorphism maps vectors of weight $\lambda$ for the adjoint action of $\Uzo$ to vectors of weight $w\lambda$. At the level of cochains, $(T_w^{-1})^*$ is induced by the map $C(\uzg,k) \rightarrow C(\uzg,k)$, $f \mapsto f \circ T_w^{-1}$, where we write $T_w^{-1}: \upB(\uzg) \rightarrow \upB(\uzg)$ to denote the evident chain map induced by $T_w^{-1}: \uzg \rightarrow \uzg$.

Lifting the coordinate functions $x_\alpha,x_\beta \in \opH^2(\uzbp,k)$ to $\opH^2(\uzg,k) \cong \g^*$, we get $(T_w^{-1})^*(x_\beta) = x_{w(\beta)} = x_\alpha$. Restricting back to $\uzbp$, we now see that we can choose a cocycle representative $f \in C^2(\uzbp,k)$ for $x_\alpha$ with support in the subspace of $\upB_2(\uzbp)$ spanned by all vectors of the form $[x_1,x_2]$ with $x_1,x_2 \in \uzo u_\zeta(e_\alpha) = T_w(\uzo u_\zeta(e_\beta))$.

\begin{proposition} \label{proposition:nilpotentaction}
Let $V$ be a finite-dimensional $\uzbp$-module. Let $\alpha \in \Phi^+$, and suppose that $V$ is injective (equivalently, projective) for the root subalgebra $u_\zeta(e_\alpha)$. Then under the cup product action of $\opH(\uzbp,k)$ on $\opH^\bullet(\uzbp,V)$, $x_\alpha$ acts nilpotently on $\opH^\bullet(\uzbp,V)$, that is, for each fixed $z \in \opH^\bullet(\uzbp,V)$, $x_\alpha^r \cup z = 0$ for all $r \gg 0$.

\end{proposition}

\begin{proof}
First, since $\opH^\bullet(\uzbp,V) = \opH^\bullet(\uzup,V)^{\uzo}$, it suffices to show that the left cup product action (equivalently, the right Yoneda product action) of $x_\alpha$ on $\opH^\bullet(\uzup,V)$ is nilpotent. The proof now breaks down into two cases:

\bigskip

\noindent \emph{Case 1. $\alpha$ is a simple root.} Assume that the reduced expression for $w_0$ chosen in \S \ref{subsection:QEAs} begins with the simple reflection $s_\alpha$, so that $E_{\gamma_1} = E_\alpha$. (Replacing one reduced expression for $w_0$ by another results in a graded $B^+$-automorphism of the ring $\opH(\uzbp,k) \cong S(\fraku^{+*})$. Any such automorphism must map $x_\alpha$ to a non-zero scalar multiple of itself, so there is no harm in making the above assumption on $w_0$.) Since $V$ is injective for $u_\zeta(e_\alpha)$, $\opH^\bullet(u_\zeta(e_\alpha),V) = \opH^0(u_\zeta(e_\alpha),V) \cong V^{u_\zeta(e_\alpha)}$ is finite-dimensional. In particular, the cup product action of $x_\alpha \in \opH^2(\uzbp,k)$ on $\opH^\bullet(u_\zeta(e_\alpha),V)$ must be nilpotent.

Now fix $1 \leq m < N$. Let $A$ be the subalgebra of $\uzbp$ generated by $\set{E_{\gamma_1},E_{\gamma_2},\ldots,E_{\gamma_{m+1}}}$, and let $B$ be the subalgebra of $\uzbp$ generated by $\set{E_{\gamma_1},\ldots,E_{\gamma_m}}$. The algebras $A$ and $B$ are left coideal subalgebras for the twisted Hopf algebra structure on $\uzbp$ by Corollary \ref{corollary:twistedcoideals}, and $B$ is normal in $A$ by Lemma \ref{lemma:commutationrelations}. We prove by induction on $m$ that the cup product action of $x_\alpha$ on $\opH^\bullet(A,V)$ is nilpotent, the case $m = 1$ already having been established. The main tool for the induction argument is the LHS spectral sequence
\begin{equation} \label{eq:LHSspecseq2}
E_1^{i,j} = \Hom_{A//B}(\upB_i(A//B),\opH^j(B,V)) \Rightarrow \opH^{i+j}(A,V).
\end{equation}
The algebra $A$ is free as a right $B$-module by the description of the PBW-basis for $\uzup$, so we can use the construction of \eqref{eq:LHSspecseq2} presented in \S \ref{subsection:LHSspecseq}.

Fix a cocycle representative $f \in \Hom_{\uzbp}(\upB_2(\uzbp),k)$ for $x_\alpha$ as in \S \ref{subsubsection:one}. Let $f^{\times r}$ denote the $r$-fold cup product $f \cup f \cup \cdots \cup f$. To show that the cup product action of $x_\alpha$ on $\opH^\bullet(A,V)$ is nilpotent, it suffices to show that for an arbitrary cocycle $g \in C^n(A,V)$, the iterated cup product $(f^{\times r}) \cup g \in C^{n+2r}(A,V)$ is a coboundary in $C(A,V)$ for all $r$ sufficiently large. To prove that the cocycle $(f^{\times r}) \cup g$ is a coboundary, we show that its image in the $E_1$-page of \eqref{eq:LHSspecseq2} is zero.

Let $g \in F^p C^n(A,V)$ be a cocycle. Then for all $r \geq 1$, $(f^{\times r}) \cup g \in F^pC^{n+2r}(A,V)$. It may happen that $(f^{\times r}) \cup g \in F^sC(A,V)$ for some $s > p$. We claim that for all $r \geq 1$, if $(f^{\times r}) \cup g \neq 0$, then $(f^{\times r}) \cup g \notin F^{n+1}C(A,V)$. In particular, $(f^{\times r}) \cup g \notin F^s C(A,V)$ for any $s > n$ unless the cochain is identically zero. Indeed, if $(f^{\times r}) \cup g \in F^{n+1}C^{n+2r}(A,V)$, then
\[
\left( (f^{\times r})\cup g \right) ([x_1,\ldots,x_{2r},x_{2r+1},\ldots,x_{n+2r}]) \qquad (x_i \in A)
\]
must be identically zero whenever one of $x_{2r},x_{2r+1},\ldots,x_{n+2r}$ is in $K = AB_\varepsilon$. By \eqref{eq:cupproductcocycle},
\[
\left( (f^{\times r})\cup g \right) ([x_1,\ldots,x_{n+2r}]) = (f^{\times r})([x_1^{(1)},\ldots,x_{2r}^{(1)}]) \cdot x_1^{(2)} \cdots x_{2r}^{(2)} g([x_{2r+1},\ldots,x_{n+2r}]),
\]
where we have used the twisted comultiplication $\Deltatau$ for $\uzbp$. From the description of the cocycle $f$ given in \S \ref{subsubsection:one}, the term $(f^{\times r})([x_1^{(1)},\ldots,x_{2r}^{(1)}])$ vanishes identically unless $x_{2r}^{(1)} \in \uzo u_\zeta(e_\alpha)$. But by Corollary \ref{corollary:twistedcoideals}, $x_{2r}^{(1)} \in \uzo u_\zeta(e_\alpha)$ only if $x_{2r} \in K$. This proves that $(f^{\times r}) \cup g \in F^{n+1}C(A,V)$ if and only if $(f^{\times r}) \cup g = 0$.

Now, replacing $g$ if necessary by $(f^{\times s}) \cup g$ for some $s \in \N$, we may assume that $(f^{\times r})\cup g$ has nonzero image in $F^pC(A,V)/F^{p+1}C(A,V)$ for all $r \geq 1$. Then for all $r \geq 1$, the image of the cocycle $(f^{\times r}) \cup g$ in the $E_1$-page of \eqref{eq:LHSspecseq2} is the map $\eta([(f^{\times r})\cup g]) \in \Hom_{A//B}(\upB_p(A//B),\opH^{n+2r-p}(B,V))$. For fixed $x_1,\ldots,x_p \in A//B$, representative cocycles for
\begin{align*}
\eta([(f^{\times r})\cup g])&([x_1,\ldots,x_p]) \in \opH^{n+2r-p}(B,V) \qquad \text{and for} \\
\eta([g])&([x_1,\ldots,x_p]) \in \opH^{n-p}(B,V)
\end{align*}
are defined by \eqref{eq:representativecocycle}. Comparing the representative cocycles, it is straightforward to check that
\begin{equation*}
\eta([(f^{\times r})\cup g])([x_1,\ldots,x_p]) = x_\alpha^r \cdot \eta([g])([x_1,\ldots,x_p]).
\end{equation*}
By induction, $x_\alpha$ acts nilpotently on $\opH^\bullet(B,V)$. Since $A//B \cong u_\zeta(e_{\gamma_m})$ is finite-dimensional, there exists $r \in \N$ such that $\eta([(f^{\times r})\cup g])([x_1,\ldots,x_p]) = 0$ for \emph{any} fixed $x_1,\ldots,x_p \in A//B$. This proves that, for all $r$ sufficiently large, the image of $(f^{\times r}) \cup g$ in $E_1^{p,n+2r-p}$ is zero, hence that $(f^{\times r})\cup g$ is a cocycle in $C^{n+2r}(A,V)$.

\bigskip

\noindent \emph{Case 2: $\alpha$ is not a simple root.} Write $\Phi^+ = \set{\gamma_1,\ldots,\gamma_N}$ with $\gamma_i = w_i(\beta_i)$ as in \S \ref{subsection:QEAs}, and suppose that $\alpha = \gamma_m$, $1 < m < N$. Set $w = w_m$. Let $A'$ be the subalgebra of $\uzup$ generated by the root vectors $\set{E_{\gamma_m},\ldots,E_{\gamma_N}}$. Then $A:=T_{w}^{-1}(A')$ is the subalgebra of $\uzup$ generated by the set
\[
S = \set{E_{\beta_m}=T_{w_m}^{-1}(E_{\gamma_m}),T_{w_m}^{-1}(E_{\gamma_{m+1}}),\ldots,T_{w_m}^{-1}(E_{\gamma_N})}.
\]
Make $V$ into an $A$-module by defining $a.v = T_{w}(a)v$ for all $a \in A$ and $v \in V$. Write ${}^wV$ for $V$ thus considered as an $A$-module. Then ${}^wV$ is projective as a module for $u_\zeta(e_{\beta_m})$, the root subalgebra corresponding to the simple root $\beta_m$.

Note that the elements in the set $S$ are precisely the first $N-m+1$ root vectors for $U_\zeta^+$ corresponding to the reduced expression for $w_0$ beginning with the reduced word $s_{\beta_m}s_{\beta_{m+1}} \cdots s_{\beta_N}$. Then by Case 1, the right Yoneda product action of $x_{\beta_m} \in \opH(\uzbp,k)$ on $\opH^\bullet(A,{}^wV)$ is nilpotent. Applying the map $(T_w^{-1})^*: \opH^\bullet(A,{}^wV) \rightarrow \opH^\bullet(A',V)$, we get that the right Yoneda product action of $(T_w^{-1})^*(x_{\beta_m}) = x_{\gamma_m} = x_\alpha$ on $\opH^\bullet(A',V)$ is nilpotent. The algebra $A'$ is a left coideal subalgebra of $\uzbp$ (with respect to the usual Hopf algebra structure) by Corollary \ref{corollary:leftcoideals}, so by the discussion in \S \ref{subsection:cohomologyproducts}, the right Yoneda product action of $\opH^{2\bullet}(\uzbp,k)$ on $\opH^\bullet(A',k)$ is the same as the left cup product action. So the left cup product action of $x_\alpha \in \opH(\uzbp,k)$ on $\opH^\bullet(A',V)$ is nilpotent.

The remainder of the proof in Case 2 now proceeds by a spectral sequence induction argument similar to that in Case~1. For the induction argument, the algebras $A$ and $B$ now have the form $A = \subgrp{E_{\gamma_{j+1}},\ldots,E_{\gamma_N}}$, $B = \subgrp{E_{\gamma_j},\ldots,E_{\gamma_N}}$ with $1 \leq j \leq m$. These algebras are left coideal subalgebras of $\uzbp$ by Corollary \ref{corollary:leftcoideals}. Choose a cocycle representative $f$ for $x_\alpha$ as in \S \ref{subsubsection:two}, and then proceed as in Case 1. The details are left to the reader.
\end{proof}

\subsection{Support varieties for the small Borel subalgebra}

Proposition \ref{proposition:nilpotentaction} can be strengthened by the observation:

\begin{lemma} \label{lemma:finitelygenerated}
Let $V$ be a finite-dimensional $\uzbp$-module. Then the Yoneda product makes $\opH^\bullet(\uzbp,V)$ a finitely-generated right $\opH^\bullet(\uzbp,k)$-module.
\end{lemma}

\begin{proof}
First we prove the corresponding result for $\uzup$. Let $V$ be a finite-dimensional $\uzup$-module, and let $V = V_0 \supset V_1 \supset \cdots \supset V_r \supset V_{r+1} = 0$ be a composition series for $V$. Then for all $0 \leq i \leq r$, $V_i/V_{i+1} \cong k$, the unique simple $\uzup$-module. Now by a standard argument using induction on the dimension of $V$ and the long exact sequence in cohomology, $\opH^\bullet(\uzup,V)$ is a finite module over $\opH^\bullet(\uzup,k)$.

Now assume that $V$ is a $\uzbp$-module. Then $\opH^\bullet(\uzbp,V) = \opH^\bullet(\uzup,V)^{\uzo}$. The space $\opH^\bullet(\uzup,k)$ is a ring under the Yoneda product, and it is finitely-generated over the subring $\opH^\bullet(\uzup,k)^{\uzo} \cong \opH^\bullet(\uzbp,k)$, cf.\ \cite[\S 2.5]{Ginzburg:1993}. The commutative ring $\uzo$ act compatibly and completely reducibly on both $\opH^\bullet(\uzup,V)$ and the Noetherian ring $\opH^\bullet(\uzup,k)$, hence by \cite[Lemma 1.13]{Friedlander:1986}, $\opH^\bullet(\uzbp,V)$ is a finite module over the Noetherian ring $\opH^\bullet(\uzbp,k)$.
\end{proof}

\begin{corollary} \label{corollary:globallynilpotent}
Let $V$ be a finite-dimensional $\uzbp$-module. Let $\alpha \in \Phi^+$, and suppose that $V$ is injective for the root subalgebra $u_\zeta(e_\alpha)$. Then there exists $r \in \N$ such that the cup product action of $x_\alpha^r \in \opH^{2r}(\uzbp,k)$ on $\opH^\bullet(\uzbp,V)$ is identically zero.
\end{corollary}

\begin{proof}
Apply Proposition \ref{proposition:nilpotentaction} and Lemma \ref{lemma:finitelygenerated}, and the commutativity of $\opH(\uzbp,k)$.
\end{proof}

Now we state and prove the main theorem of this section:

\begin{theorem} \label{theorem:rootvectorcriterion}
Let $M$ be a finite-dimensional $\uzbp$-module, and let $\alpha \in \Phi^+$. Then the root vector $e_\alpha \in \fraku^+$ is an element of $\Vbp(M)$ if and only if $M$ is not projective for $u_\zeta(e_\alpha)$.
\end{theorem}

\begin{proof}
First suppose that $M$ is projective (equivalently, injective) for the root subalgebra $u_\zeta(e_\alpha)$. Then $M$ is projective for the Hopf algebra $u_\zeta(\frakb_\alpha^+) := \subgrp{E_\alpha,K_\alpha} \subset \uzbp$ by Lemma \ref{lemma:borelunipotentequivalence}, hence $\Hom_k(M,M) \cong M \otimes M^*$ is also projective for $u_\zeta(\frakb_\alpha^+)$. Applying Lemma \ref{lemma:borelunipotentequivalence} again, $\Hom_k(M,M)$ is injective for $u_\zeta(e_\alpha)$. Now Corollary \ref{corollary:globallynilpotent} asserts that $x_\alpha^r \in J_{\uzbp}(M)$ for some $r \in \N$. It follows that $e_\alpha \notin \Vbp(M)$.

Now suppose that $e_\alpha \notin \Vbp(M)$. The restriction map $\opH(\uzbp,k) \rightarrow \opH(u_\zeta(\frakb_\alpha^+),k) = k[x_\alpha]$ is surjective, and induces a closed embedding
\[
\calV:=\calV_{u_\zeta(\frakb_\alpha^+)}(M) \hookrightarrow \Vbp(M).
\]
The variety $\calV$ is one-dimensional, and its image in $\Vbp(M)$ is either zero or the line spanned by the root vector $e_\alpha$. If $e_\alpha \notin \Vbp(M)$, then $\calV = \set{0}$. By a standard argument (cf.\ \cite[Proposition 1.5]{Friedlander:1987}), this implies that $M$ is projective over the Hopf algebra $u_\zeta(\frakb_\alpha^+)$.
\end{proof}

\begin{corollary} \label{corollary:rootvectorcriterion}
Let $M$ be a finite-dimensional $\uzb$-module, and let $\alpha \in \Phi^+$. Then the root vector $f_\alpha \in \fraku$ is an element of $\Vb(M)$ if and only if $M$ is not projective for $u_\zeta(f_\alpha)$.
\end{corollary}

\begin{proof}
Use the fact that $\omega: \uzb \rightarrow \uzbp$ is an automorphism.
\end{proof}

\begin{remark}
We would like to generalize Corollary \ref{corollary:rootvectorcriterion} from root vectors to linear combinations of elements of $\fraku$, but it is not clear what the corresponding subalgebras of $\uzb$ should be. We would also like to generalize Corollary \ref{corollary:rootvectorcriterion} from $\Vb(M)$ to $\Vg(M)$, but this seems prohibitively difficult under the present strategy, since we have less control over the cocycle representatives for the $x_\alpha$ after lifting them to $\opH(\uzg,k)$.
\end{remark}

\section{Applications} \label{section:applications}

\subsection{A geometric proof of Theorem \ref{theorem:injectivitycriterion}} \label{subsection:secondproof}

We can now provide a second, geometric proof of the $r=0$ case of Theorem \ref{theorem:injectivitycriterion}. The argument is formally similar to that in \cite[\S 3.4]{Friedlander:1987}. We continue to impose the conditions of Assumption \ref{assumption:ellandp}.

Let $M$ be a finite-dimensional rational $\Uzo \uzg$-module, and suppose that $M$ is injective for the Borel subalgebras $\uzb$ and $\uzbp$. Then $V:= M \otimes M^*$ is injective for $\uzb$ and $\uzbp$. By Proposition \ref{proposition:injectivefiltration}, the injectivity of $V$ for $\uzbp$ implies that $V$ admits a filtration by $\Uzo \uzg$-submodules with factors of the form $\wh{Z}_0'(\lambda)$, $\lambda \in X$. By Frobenius reciprocity and the exactness of induction from $\Uzo \uzb$ to $\Uzo \uzg$, $\opH^\bullet(\uzg,\wh{Z}_0'(\lambda)) \cong \opH^\bullet(\uzb,\lambda)$. This isomorphism is compatible with the cup product action of $\opH(\uzg,k)$. Then for each $\lambda \in X$, $\ker(\res) \subseteq J_{\uzg}(\wh{Z}_0'(\lambda))$, where $\res: \opH(\uzg,k) \rightarrow \opH(\uzb,k)$ denotes the cohomological restriction map. It now follows from the long exact sequence in cohomology and by induction on the number of factors $\wh{Z}_0'(\lambda)$, $\lambda \in X$, in the $\Uzo \uzg$-filtration of $V$ that $\ker(\res)$ is contained in the radical of the ideal $J_{\uzg}(M)$. This implies that $\Vg(M) \subset \fraku$. Since $\omega: \uzb \stackrel{\sim}{\rightarrow} \uzbp$ and $(\omega)^*: \opH(\uzb,k) \stackrel{\sim}{\rightarrow} \opH(\uzbp,k)$ are isomorphisms, we obtain by symmetry that the injectivity of $V$ for $\uzb$ implies that $\Vg(M) \subset \fraku^+$. Since $\fraku \cap \fraku^+ = \set{0}$, the injectivity of $M$ for $\uzb$ and $\uzbp$ implies that $\Vg(M) = \set{0}$, hence that $M$ is injective for $\uzg$.

We have thus geometrically reduced the problem of Theorem \ref{theorem:injectivitycriterion} to the statement of Theorem \ref{theorem:injectiveBorel}. Now, since $M$ is a rational $\Uzo$-module, the support variety $\Vb(M)$ is a $T$-stable subvariety of $\fraku$. Similarly, $\Vbp(M)$ is a $T$-stable subvariety of $\fraku^+$. By Theorem \ref{theorem:rootvectorcriterion} and Corollary \ref{corollary:rootvectorcriterion}, neither support variety contains any root vectors. But any non-zero $T$-stable subvariety of $\fraku$ (resp.\ $\fraku^+$) must contain a root vector. We conclude that $\Vb(M) = \set{0} = \Vbp(M)$, hence that the injectivity of $M$ for each root subalgebra $u_\zeta(e_\alpha)$, $u_\zeta(f_\alpha)$, $\alpha \in \Phi^+$, implies the injectivity of $M$ for $\uzb$ and $\uzbp$.

\subsection{Support varieties for the small quantum group} \label{subsection:Gsupportvarieties}

We continue to impose the conditions of Assumption \ref{assumption:ellandp}. The next result and its proof are a translation to the quantum setting of \cite[\S 1.2]{Friedlander:1986b}.

\begin{theorem} \label{theorem:Gorbitbvariety}
Let $M$ be a finite-dimensional $\Uzg$-module. Then $\Vg(M) = G \cdot \Vb(M)$, where $G \cdot \Vb(M)$ is the orbit of $\Vb(M) \subset \calN$ under the adjoint action of $G$.
\end{theorem}

\begin{proof}
Make the identifications $\opH(\uzg,k) = k[\calN]$ and $\opH(\uzb,k) = S(\fraku^*)$, and define ideals
\begin{align*}
S(\fraku^*) \supset I_M &:= J_{\uzb}(M), \\
S(\g^*) \supset J_M &:= J_{\uzg}(M), \\
S(\g^*) \supset K_M &:= \set{ f \in S(\g^*): \, \forall \, g \in G, (g \cdot f)|_\fraku \in I_M}, \\
S(\g^*) \supset L_M &:= \set{ f \in S(\g^*): \, \forall \, g \in G, (g \cdot f)|_\fraku \in \sqrt{I_M}}.
\end{align*}
Then $I_M$ and $J_M$ are the ideals defining $\Vb(M)$ and $\Vg(M)$, respectively, and $L_M$ is the ideal of functions defining the closed subvariety $G \cdot \Vb(M)$ of $\calN$. To prove the theorem it suffices to show that $L_M = \sqrt{J_M}$. (Note that $L_M$ is already a radical ideal.) The inclusion $\sqrt{J_M} \subseteq L_M$ is easy. Indeed, since the restriction map $\opH(\uzg,k) \rightarrow \opH(\uzb,k)$ is just the restriction of functions, it is in particular surjective, so $\Vb(M)$ identifies naturally with a closed suvariety of $\Vg(M)$. The support variety $\Vg(M)$ is $G$-stable, so also $G \cdot \Vb(M) \subseteq \Vg(M)$. But this last inclusion is equivalent to the ideal inclusion $\sqrt{J_M} \subseteq L_M$. So it remains to show that $L_M \subseteq \sqrt{J_M}$. We do this by proving the inclusions $L_M \subseteq \sqrt{K_M}$ and $K_M \subseteq \sqrt{J_M}$.

Given $f \in L_M$, let $I_f \subset S(\g^*)$ be the ideal generated by $\set{g \cdot f: g \in G}$. Since $S(\g^*)$ is a Noetherian ring, $I_f$ is finitely generated, say by the set $\set{g_1\cdot f, \ldots, g_j \cdot f}$. For each $1 \leq i \leq j$, there exists $a_i \in \N$ such that $(g_i \cdot f)^{a_i}|_\fraku \in I_M$. Now fix $A \in \N$ with $A \geq \min \set{j \cdot a_i}$, and let $g \in G$. There exist $\lambda_i \in S(\g^*)$ such that $g \cdot f = \sum_{i=1}^j \lambda_i(x_i \cdot f)$. Then $g \cdot f^A = (g \cdot f)^A = ( \sum_{i=1}^j \lambda_i(x_i \cdot f) )^A$. Expanding the last expression and restricting it to $\fraku$, we see that each summand is an element of $I_M$, because in the expansion of $(\sum_{i=1}^j \lambda_i(x_i \cdot f))^A$, each summand is divisible by $(x_i \cdot f)^{a_i}$ for some $1 \leq i \leq j$. Then $f^A \in K_M$, so $L_M \subseteq \sqrt{K_M}$.

It remains to show $K_M \subseteq \sqrt{J_M}$. By the discussion in \S \ref{subsection:spectralsequences}, there exists a commutative diagram
\begin{equation} \label{eq:commdiagram}
\begin{split}
\xymatrix{
K_M \ar@{^(->}[r] \ar@{->}[d] & S^\bullet(\g^*) \ar@{->}[d] \ar@{->}[r] &
\opH^{2\bullet}(\uzg,k) \ar@{->}[d]^{\sim} \ar@{->}[r] & \opH^{2\bullet}(\uzg,\Hom_k(M,M)) \ar@{->}[d] \\
\ind_B^G I_M \ar@{->}[r] & \ind_B^G S^{\bullet}(\fraku^*) \ar@{->}[r]^(.4){\sim} & \ind_B^G \opH^{2\bullet}(\uzb,k) \ar@{->}[r] & \ind_B^G \opH^{2\bullet}(\uzb,\Hom_k(M,M)).
}
\end{split}
\end{equation}
Each vertical map is induced by Frobenius reciprocity from the corresponding restriction map: $K_M \rightarrow I_M$, $S^\bullet(\g^*) \rightarrow S^\bullet(\fraku^*)$, etc. The third and fourth vertical maps are the edge maps for the spectral sequences $E_2(k)$ and $E_2(V)$ discussed in \S \ref{subsection:spectralsequences}, where $V:= \Hom_k(M,M)$. The maps into the right-most column are induced by the map $k \rightarrow \Hom_k(M,M)$, $1 \mapsto \id$.

Composition along the bottom row in \eqref{eq:commdiagram} is zero, because the composite map $I_M \rightarrow S^\bullet(\fraku^*) \cong \opH^{2\bullet}(\uzb,k) \rightarrow \opH^{2\bullet}(\uzb,V)$ is zero by the definition of $I_M$. This implies that the composite map along the top row and down the right-most column, from $K_M$ to $\ind_B^G \opH^{2\bullet}(\uzb,V)$, is zero. Then the image of $K_M$ in $\opH^{2\bullet}(\uzg,V)$ must have positive filtration degree with respect to the filtration on $\opH^\bullet(\uzg,V)$ coming from the spectral sequence \eqref{eq:uzgLHSspecseq}. We have $E_2^{i,j}(V) = 0$ for $i > d:=\dim(G/B)$, because $R^i \ind_B^G(-) = 0$ for $i > d$ \cite[II.4.2]{Jantzen:2003}. Then $E_\infty^{i,j}(V) = 0$ for $i>d$. Since $E_r(V)$ is a module over $E_r(k)$, we conclude that $(K_M)^{d+1}$ maps to zero in $\opH^\bullet(\uzg,V)$, hence that $K_M \subseteq \sqrt{J_M}$.
\end{proof}

\begin{corollary}
Let $M$ be a finite-dimensional $\Uz$-module. Let $\alpha_h \in \Phi^+$ be the highest positive root. Then $M$ is injective for $\uzg$ if and only if $M$ is injective for the root subalgebra $u_\zeta(f_{\alpha_h})$.
\end{corollary}

\begin{proof}
The surjectivity of the restriction map $\opH(\uzg,k) \rightarrow \opH(\uzb,k)$ implies that there exists a closed embedding $\Vb(M) \hookrightarrow \Vg(M)$. If $M$ is injective for $\uzg$, then $\Vg(M) = \set{0}$, so necessarily $\Vb(M) = \set{0}$. Then $M$ is injective for $u_\zeta(f_{\alpha_h})$ by Corollary \ref{corollary:rootvectorcriterion}.

Conversely, suppose $M$ is not injective for $\uzg$. Then $\Vg(M) \neq \set{0}$ \cite[Proposition 2.4]{Feldvoss:2009}. (The last statement uses the fact that for every pair of finite-dimensional $\uzg$-modules $M$ and $N$, the space $\Ext_{\uzg}^\bullet(M,N)$ is a finitely-generated $\opH^\bullet(\uzg,k)$-module under the cup product, cf. \cite{Bendel:2009} or \cite[Theorem 4.24]{Drupieski:2009}.) By Theorem \ref{theorem:Gorbitbvariety}, this implies that $\Vb(M) \neq \set{0}$. Now, since $M$ is a $\Uzb$-module, $\Vb(M)$ is a non-zero closed $B$-stable subvariety of $\fraku$. In particular, $\Vb(M)$ is a $T$-stable closed subvariety of $\fraku$, so it must contain a root vector $f_\alpha \in \fraku$, for some $\alpha \in \Phi^+$. If $\beta \in \Phi^+$ and $\alpha + \beta \in \Phi^+$, then $f_{\alpha + \beta} \in \ol{U_\beta \cdot f_\alpha}$, the $U_\beta$-orbit closure of $f_\alpha$. Since $\Vb(M)$ is closed, it follows that $f_{\alpha_h} \in \Vb(M)$, hence that $M$ is not injective for $u_\zeta(f_{\alpha_h})$ by Corollary \ref{corollary:rootvectorcriterion}.
\end{proof}

\bibliographystyle{amsplain}
\bibliography{injectivity_criterion}

\end{document}